\documentclass[a4paper,11pt,reqno,noindent]{article}
\usepackage[affil-it]{authblk}
\usepackage{amsmath,amsfonts,amssymb,amsthm} 
\usepackage[dvips]{graphicx} 
\usepackage{psfrag}
\usepackage[english]{babel}
\usepackage{newlfont}
\usepackage{fullpage}
\usepackage{esint}
\usepackage{enumerate}
\usepackage{hyperref}
\newtheorem{lemma}{Lemma}
\newtheorem{prop}{Proposition}
\newtheorem{corollary}{Corollary}
\newtheorem{theo}{Theorem}
\newtheorem{rem}{Remark}

\theoremstyle{definition}
\newtheorem{definition}{Definition}
\newtheorem{assumption}{Assumption}

\newcommand{\dist}{\mathrm{dist}}
\newcommand{\R}{\mathbb{R}}
\newcommand{\Z}{\mathbb{Z}}
\newcommand{\N}{\mathbb{N}}
\newcommand{\B}{\mathbb{B}}
\newcommand{\e}{\varepsilon}
\renewcommand{\aa}{\boldsymbol a}
\newcommand{\ee}{e}
\newcommand{\bb}{\mathrm b}
\mathchardef\emptyset="001F

\newcommand{\expec}[1]{\left\langle #1 \right\rangle}

\newcommand{\step}[1]{\noindent \textit{Step} #1.}

\begin{document} 
\title{Moment bounds for the corrector in stochastic homogenization of
  a percolation model}
\author{Agnes Lamacz%
  \thanks{agnes.lamacz@tu-dortmund.de}}
 \affil{Technische Universit\"at Dortmund, Fakult\"at f\"ur
   Mathematik, Germany}
\author{Stefan Neukamm\thanks{stefan.neukamm@wias-berlin.de}}
\affil{Weierstrass Institute for Applied Analysis and Stochastics, Berlin, Germany}
\author{Felix Otto\thanks{felix.otto@mis.mpg.de}}  
\affil{Max Planck Institute for Mathematics in the Sciences, Leipzig, Germany}
\date{June 22, 2014\footnote{This is a revised version of WIAS Preprint No. 1836 (2013)}
}
\maketitle

\begin{abstract}
  We study the corrector equation in stochastic homogenization
  for a simplified Bernoulli percolation model on $\Z^d$, $d>2$. The model is obtained from the
  classical $\{0,1\}$-Bernoulli bond percolation by
  conditioning all bonds parallel to the first coordinate
  direction to be open. As a  main result we prove (in fact for a slightly more
  general model) that stationary correctors exist and that all
  finite moments of the corrector are bounded. This extends a previous
  result in \cite{GO1}, where uniformly elliptic conductances are
  treated, to the degenerate case. With regard to the associated random conductance
  model, we obtain as a side result  that the corrector not only grows
  sublinearly, but slower than any polynomial rate. Our argument
  combines a quantification of ergodicity by means of a Spectral Gap
  on Glauber dynamics with regularity estimates on the gradient of the
  elliptic Green's function.
  \medskip

  \noindent {\bf Keywords:} 
  quantitative stochastic homogenization, percolation, corrector.
\end{abstract}

\section{Introduction}
We consider the lattice graph $(\Z^d,\B^d)$, $d>2$, where $\B^d$ denotes the set of nearest-neighbor
edges. Given a stationary and ergodic probability measure
$\expec{\cdot}$ on $\Omega$ -- the space of conductance fields
$\aa:\B^d\to[0,1]$ -- we study the \textit{corrector equation} from stochastic
homogenization, i.e.~the elliptic difference equation
\begin{equation}\label{eq:cor-equation-phys}
  \nabla^*(\aa(\nabla\phi+e))=0,\qquad x\in\Z^d.
\end{equation}
Here,  $\nabla$ and $\nabla^*$ denote discrete versions of the continuum gradient
and (negative) divergence, cf. Section~\ref{S2}, and $e\in\R^d$ denotes
a vector of unit length, which is fixed throughout the paper. The corrector equation \eqref{eq:cor-equation-phys} emerges in the homogenization of discrete elliptic equations with random
coefficients:  For random conductances that are stationary and
ergodic (with respect to the shifts $\aa(\cdot)\mapsto\aa(\cdot+z)$,
$z\in\Z^d$, cf. Section~\ref{S2}), and under the assumption of uniform
ellipticity (i.e.~there exists $\lambda_0>0$ such
that $\aa\geq \lambda_0$ on $\B^d$ almost surely), a classical result from stochastic homogenization (e.~g. see \cite{Kozlov-87, Kunnemann-83}) shows that the effective behavior of $\nabla^*\aa\nabla$ on large length scales is captured by
the homogenized elliptic operator $\nabla^*\aa_{\hom}\nabla$ where
$\aa_{\hom}$  is  a
deterministic, symmetric and positive definite $d\times d$ matrix. It is characterized by
the minimization problem
\begin{equation}\label{eq:min}
  e \cdot\aa_{\hom}e =\inf\limits_{\varphi}\expec{(e +\nabla\varphi)\cdot\aa(e +\nabla\varphi)},
\end{equation}
where the infimum is taken over random fields
$\varphi$ that are $\expec{\cdot}$-stationary in the sense of $\varphi(\aa,x+z)=\varphi(\aa(\cdot+z),x)$ for all
$x,z\in\Z^d$ and $\expec{\cdot}$-almost every $\aa\in\Omega$.
Minimizers to  \eqref{eq:min} are called \textit{stationary
  correctors} and are characterized as the stationary solutions to the corrector
equation \eqref{eq:cor-equation-phys}. Due to the lack of a Poincar\'e inequality for $\nabla$ on the infinite dimensional space of
stationary random fields, the elliptic operator $\nabla^*\aa\nabla$ is highly degenerate
and the minimum in \eqref{eq:min} may not be obtained in general. In
fact, it is known to fail generally for $d=2$. The only existence result of a stationary corrector (in dimensions
$d>2$) has been obtained
recently in \cite{GO1} by Gloria and the third author under the
assumption that the $\aa$'s are uniformly elliptic, and that
$\expec{\cdot}$ satisfies a Spectral Gap Estimate, which is in
particular the case for independent and identically distributed coefficients. They also show that $\expec{|\phi|^p}\lesssim 1$ for all $p<\infty$.

\medskip 
The goal of the present paper is to extend this result to the case of
conductances with degenerate ellipticity.  To be definite, consider the probability
measure $\expec{\cdot}_{\lambda}$ constructed by the following procedure:
\begin{equation}\label{eq:modbernoulli}
  \begin{aligned}
    &\text{Take the classical $\{0,1\}$-Bernoulli-bond percolation on
      $\B^d$
      with parameter $\lambda\in(0,1]$}\\
    &\text{and declare all bonds parallel to the coordinate direction
      $e_1$ to be open.}
  \end{aligned}
\end{equation}
(We adapt the convention
to call a bond ``open'' if the associated coefficient is ``$1$'',
while a bond is ``closed'' if the associated coefficient is
``$0$''. The parameter $\lambda$ denotes the probability that a bond
is ``open''). As for $d$-dimensional
Bernoulli percolation, $\expec{\cdot}_\lambda$ describes a random
graph of open bonds, which is locally
disconnected with positive probability.  However, as a
merit of the modification, any two vertices in the random graph are
almost surely connected by some open path. As a main result we show that
\eqref{eq:cor-equation-phys} admits a stationary solution, all finite
moments of which are bounded:

\paragraph{Theorem (main result).} Let $d>2$ and $\lambda\in(0,1]$. There exits $\phi:\Omega\times\Z^d\to\R$ such
that for $\expec{\cdot}_\lambda$ almost every $\aa\in\Omega$ we have 
\begin{itemize}
\item  $\phi(\aa,\cdot)$ solves \eqref{eq:cor-equation-phys},
\item  $\phi(\aa,\cdot+z)=\phi(\aa(\cdot+z),\cdot)$ for all $z\in\Z^d$,
\end{itemize}
and
  \begin{equation*}
    \forall p<\infty\,:\qquad
    \expec{|\phi|^p}_\lambda^{\frac{1}{p}}\leq C.
  \end{equation*}
  Here $C$ denotes a constant that only depends on $p$, $\lambda$ and $d$.
\medskip

The modified Bernoulli percolation model $\expec{\cdot}_\lambda$ fits into a
slightly more general framework that we introduce in Section~\ref{S2}
below, cf. Lemma~\ref{L:regul-perc-model}. The result above will then
follow as a special case of Theorem~\ref{T1} stated below. 
\medskip


\medskip

{\bf Relation to stochastic
  homogenization.} Consider the decaying solution $u_\e:\Z^d\to\R$ to the
equation
\begin{equation*}
  \nabla^*(\aa\nabla u_\e)(\cdot)=\e^2f(\e \cdot)\qquad\text{in }\Z^d,
\end{equation*}
where $f:\R^d\to\R$ is smooth, compactly supported, and $\aa$ is
distributed according to $\expec{\cdot}_\lambda$. Classical results of
stochastic
homogenization (see \cite{Papanicolaou-Varadhan-79}, \cite{Kozlov-79},
\cite{Kunnemann-83}) show that for almost every $\aa\in\Omega$ the
(piecewise constant interpolation of the) rescaled 
function $u_\e(\tfrac{\cdot}{\e})$ converges as
$\e\downarrow 0$ to the
unique decaying solution $u_{\hom}:\R^d\to\R$ of the deterministic elliptic equation
\begin{equation*}
  -\nabla\cdot(\aa_{\hom}\nabla u_{\hom})=f\qquad\text{in }\R^d.
\end{equation*}
Moreover, a formal two-scale expansion suggests that
\begin{equation}\label{eq:22}
  u_\e(x)\approx u_{\hom}(\e x)+\e\sum_{j=1}^d\phi_{j}(x)\partial_ju_{\hom}(\e x),
\end{equation}
where $\phi_j$ denotes the (stationary) solution of
\eqref{eq:cor-equation-phys} for $e=e_j$ -- the $j$th coordinate
direction. The question how to \textit{quantify} the errors emerging
in this limiting process  is rather subtle. Note that in the case of deterministic periodic homogenization, the
good compactness properties of the $d$-dimensional ``reference cell of periodicity''
yield a natural starting point for estimates. In contrast, in the
stochastic case the reference cell has to be replaced by the probability space
$(\Omega,\expec{\cdot})$, which has infinite dimensions and thus most
``periodic technologies'' break down. Nevertheless,
estimates for the homogenization error
$\|u_\e-u_{\hom}\|$ and related quantities have been obtained by
\cite{Yurinskii-76,  Caputo-Ioffe-03, Bourgeat-04, Conlon-Spencer-13},
see also \cite{Cafarelli-Souganidis-10,Armstrong-Smart-13} for recent
results on fully nonlinear elliptic equations or equations in
non-divergence form.

While the asymptotic result of stochastic homogenization holds for general stationary and ergodic coefficients (at least in the uniformly elliptic
case), the derivation of error estimates requires a quantification of
ergodicity. In a series of papers (see \cite{GNO1,GNO3, GO1,GO2}) two of the authors and Gloria developed a
quantitative theory for the corrector equation
\eqref{eq:cor-equation-phys} (and regularized versions) based on the
assumption that the underlying statistics satisfies a Spectral Gap
Estimate (SG) for a Glauber dynamics on the coefficient fields. This
assumption is satisfied e.g.  in the case of independent and identically distributed
(i.~i.~d.~coefficients). In \cite{GO1, GNO1}  moment bounds for the corrector, similar to the
one in the present paper, have been obtained. These  bounds are
at the basis of various optimal estimates; e.g. \cite{GNO1} contains a complete and optimal
analysis of the approximation of $\aa_{\hom}$ via periodic
representative volume elements, and \cite{GNO3} establishes optimal
estimates for the homogenization error and the expansion in \eqref{eq:22}.

While in the works mentioned above it is always assumed that the coefficients are
\textit{uniformly elliptic}, i.e. $\aa\in[\lambda_0,1]^{\B^d}$ for
some fixed $\lambda_0>0$, in the present paper we derive moment bounds for a
model with \textit{degenerate} elliptic coefficients. As in \cite{GO1,GNO1}, a
crucial element of our approach is an estimate on the g\textit{radient} of the \textit{elliptic Green's function} associated with
$\nabla^*\aa\nabla$. The required estimate is pointwise in $\aa$, but (dyadically) averaged
in space, and obtained by a self-contained and short argument, see Proposition~\ref{P1} below. It
extends the argument in \cite{GO1} to the degenerate elliptic case. Since in the degenerate case the elementary inequality $\lambda_0|\nabla
u|^2\leq \nabla u\cdot\aa\nabla u$ breaks down, we replace it  by a weighted, integrated
version (see Lemma~\ref{lem:coercivity-phys} below). Compared to
more sophisticated methods that e.g. rely on isoperimetric properties
of the graph, an advantage of our approach is that it only invokes
simple geometrical properties, namely spatial averages (on balls) of the inverse of the chemical distance between nearest
neighbor vertices. We believe that our
approach extends (although not in a straight-forward manner) to the case of standard supercritical Bernoulli
percolation. This is a  question that we study in a work in progress.
\medskip

{\bf Connection to random walks in random environments (RWRE).}
Although, the main motivation of our work is quantitative
homogenization, we would like to comment on the connection to invariance principles for (RWRE). In fact, there
is a strong link between stochastic homogenization and (RWRE): The operator $\nabla^*(\aa\nabla)$ generates
a stochastic process, namely the variable-speed random walk
$X=(X_{\aa}(t))_{t\geq 0}$, which is a continuous-time random walk in the random
environment $\aa$. In the early work \cite{Kipnis-Varadhan-86} (see
also \cite{Kunnemann-83}) the authors considered general stationary
and ergodic environments. For uniformly elliptic coefficients they prove an \textit{annealed
invariance principal} for $X$, saying that the law of the rescaled process $\sqrt\e
X_{\aa}(\e^{-1}t)$ weakly converges to that of a Brownian motion with
covariance matrix $2\aa_{\hom}$. In \cite{Sidoravicius-Sznitman-04}
Sidoravicius and Sznitman prove a stronger \textit{quenched invariance
principle} for $X$, which says that the convergence even holds for
almost every environment $\aa$. 

More recently, invariance principles have been obtained for
more general environments, see \cite{Biskup-11} and \cite{Kumagai-14} for recent surveys in
this direction. Most prominently, supercritical bond percolation on
$\Z^d$ has been considered: Here, the annealed result is due to
\cite{DeMasi-etal-1989}, while quenched results have been obtained in
\cite{Sidoravicius-Sznitman-04} for $d\geq 4$ and in
\cite{Berger-Biskup-07, Mathieu-Piatnitski-07}  for $d\geq 2$. See
also \cite{Andres-etal-ta, Andres-Deuschel-Slowik-ta} for recent
related results on degenerate elliptic, possibly unbounded
conductances. 

The main difficulty in proving a quenched invariance
principle compared to the annealed version is to establish a
\textit{quenched sublinear growth} property (see \eqref{eq:18} below)
for a corrector field $\chi$. The latter is closely related to the function
$\phi$ considered in Theorem~\ref{T1}, see the discussion below
Corollary~\ref{C1} for more details. In the uniformly elliptic case,
sublinearity of $\chi$ is obtained by soft arguments from
ergodic theory combined with a Sobolev embedding, see
\cite{Sidoravicius-Sznitman-04}. For supercritical Bernoulli
percolation the argument is more subtle: For $d\geq 3$ the proofs in
\cite{Sidoravicius-Sznitman-04,Berger-Biskup-07,
  Mathieu-Piatnitski-07}  use heat-kernel upper bounds (as deduced by
Barlow \cite{Barlow-04}) or other ``heat-kernel technologies'' (e.g. see
\cite{Biskup-Prescott-07, Andres-etal-ta, Andres-Deuschel-Slowik-ta}) that require a detailed
understanding of the geometry of the percolation cluster, and thus
require the use of sophisticated arguments from percolation theory
(e.g. isoperimetry, regular volume growth and comparison of chemical
and Euclidean distances). Conceptually, the use of
such fine arguments seems not to be necessary in the derivation of
quenched invariance principles. Motivated by this in
\cite{Biskup-Prescott-07} and \cite{Andres-Deuschel-Slowik-ta}
different methods are employed with a reduced  usage of heat-kernel technology. 

Our approach yields, as  a side-result,  an alternative way to achieve
this goal: The quenched sublinear growth property can easily be
obtained from the moment bound derived in Theorem~\ref{T1}. In fact,
the estimate of Theorem~\ref{T1} is stronger: As we explain in the
discussion following Corollary~\ref{C1}, our moment bounds imply that the
growth of $\chi$ is not only sublinear, but \textit{slower than any rate}, see
\eqref{eq:21}. Of course, the environment considered in the
present paper, namely the modified percolation model
$\expec{\cdot}_\lambda$, is much simpler than
supercritical Bernoulli percolation. Nevertheless, it shares some of
the ``degeneracies'' featured by percolation; e.g. for every ball
$B\subset\Z^d$ with finite radius Poincar\'e's inequality $\sum_{x\in
  B}u^2(x)\leq C(\aa,B)\sum_{\bb\in B}\aa(\bb)|\nabla u(\bb)|^2$ fails
with positive probability. Furthermore, in contrast to the above mentioned results, our 
argument requires  only mild estimates on the Green's function. More
precisely, as already mentioned, we require an estimate on the \textit{gradient} of the
\textit{elliptic} Green's function, which -- in contrast to quenched heat kernel
estimates -- can be obtained by fairly simple arguments, see
Proposition~\ref{P1}. Of course, as it is well-known, estimates on the gradient of
the elliptic Green's function can also be obtained from estimates on the associated heat kernel by an
integration in time, and a subsequent application of Caccioppoli's
inequality. In particular, heat-kernel estimates in the spirit of the one obtained by Barlow in the case of
supercritical Bernoulli percolation, see \cite[Theorem~1]{Barlow-04},
would be sufficient to make this program work. Yet, since the elliptic
estimates that we require are less sensitive to the geometry of the
graph, and thus can be obtained by simpler arguments, we opt for a self-contained proof that only relies on elliptic
regularity theory. Another interesting, and -- as we believe -- advantageous property of our approach is that (thanks to the Spectral Gap
Estimate) probabilistic and deterministic considerations are well
separated, e.g. Proposition~\ref{P1} is pointwise in $\aa$ and does not
involve the ensemble.

\medskip

{\bf Structure of the paper.} In Section~\ref{S2} we gather basic
definitions and introduce the slightly more general framework studied
in this paper. We then present the main result in the general  framework. Section~\ref{S3} is devoted to the proof of the main result: we first
discuss the general strategy of the proof and present several
auxiliary lemmas needed for the proof of the main theorem -- in
particular, the coercivity estimates, see Lemmas~\ref{lem:coercivity-phys}
and \ref{lem:coercivity-prob}, and an estimate
for the gradient of the elliptic Green's function, see
Proposition~\ref{P1}, which  play a key
role in our argument. The proof of the main result is given at the end
of Section~\ref{S3}, while the auxiliary results are proven in
Section~\ref{S:proofs}.

\medskip
Throughout this article, we use the following notation, see
Section~\ref{S2} for more details:
\begin{itemize}
\item $d$ is the dimension;
\item $\Z^d$ is the integer lattice;
\item $(\ee_1,\dots,\ee_d)$ is the canonical basis of $\Z^d$;
\item $e\in\R^d$, which appears in \eqref{eq:cor-equation-phys},
  denotes a vector of unit length and is fixed throughout the paper;
\item $\B^d:=\{\,\bb=\{x,x+e_i\}\,:\,x\in\Z^d,\,i=1,\ldots,d\,\}$ is the set of nearest neighbor bonds of $\Z^d$;
\item $B_R(x_0)$ is the cube of vertices $x\in x_0+([-R,R]\cap\Z)^d$;
\item $Q_R(x_0)$ is the cube of bonds $\bb=\{x, x+e_i\}\in\B^d$ with
  $x\in B_{R}(x_0)$ and $i\in\{1,\ldots,d\}$;
\item $|A|$ denotes the number of elements in $A\subset\Z^d$ (resp. $A\subset\B^d$).
\end{itemize}

\section{General framework}\label{S2}
In the first part of this section, we introduce the general framework
following the presentation of \cite{GNO1}: We introduce a discrete differential
calculus, the random conductance model, and finally recall the standard definitions of the
corrector and the modified corrector.

\subsection{Lattice and discrete differential calculus}
We consider the lattice graph $(\Z^d,\B^d)$, where $\B^d:=\{\,\bb=\{x,x+e_i\}\,:\,x\in\Z^d,\,i=1,\ldots,d\,\}$ denotes the
set of nearest-neighbor bonds. We write $\ell^p(\Z^d)$ and $\ell^p(\B^d)$, $1\leq p\leq\infty$, for the
usual spaces of $p$-summable (resp. bounded for $p=\infty$) functions on $\Z^d$ and $\B^d$. For $u:\Z^d\to\R$ the
\textit{discrete derivative} $\nabla u(\bb)$, $\bb\in\B^d$, is
defined by the expression
\begin{gather*}
  \nabla u(\bb):= u(y_{\bb})-u(x_{\bb}).
\end{gather*}
Here $x_{\bb}$ and $y_{\bb}$ denote the unique vertices with
$\bb=\{x_{\bb},y_{\bb}\}\in\B^d$ satisfying
$y_{\bb}-x_{\bb}\in\{e_1,\ldots,\ee_d\}$. 
We denote by $\nabla^*$ the adjoined of $\nabla$, so that we have for $F:\B^d\to\R$ 
\begin{equation*}
  \nabla^*F(x)=\sum_{i=1}^dF(\{x-e_i,x\})-F(\{x,x+e_i\}).
\end{equation*}
Furthermore, the discrete integration by parts formula reads 
\begin{equation}\label{int-by-parts}
  \sum_{\bb\in\B^d}\nabla u(\bb) F(\bb)=  \sum_{x\in\Z^d}u(x)\nabla^*F(x),
\end{equation}
and holds whenever the sums converge.
\medskip

\subsection{Random conductance field}  To each bond $\bb\in\B^d$ a \textit{conductance} $\aa(\bb)\in[0,1]$ is
attached. Hence, a \textit{configuration} of the lattice is described by a \textit{conductance field}
$\aa\in\Omega$, where $\Omega:=[0,1]^{\B^d}$ denotes the
\textit{configuration space}. Given $\aa\in\Omega$ we define the
chemical distance between vertices
$x,y\in\Z^d$ by
\begin{equation*}
  \dist_{\aa}(x,y):=\inf\left\{\,\sum_{\bb\in\pi}\aa(\bb)^{-1}\ :\
      \pi\text{ is a path from $x$ to $y$}\ \right\}\qquad(\text{where }\tfrac{1}{0}:=+\infty).
\end{equation*}
We equip $\Omega$  with the product topology (induced
by $[0,1]\subset\R$) and the usual product $\sigma$-algebra, and
describe \textit{random configurations} by means of a probability
measure on $\Omega$, called the \textit{ensemble}. The associated expectation is denoted by
$\expec{\cdot}$.
\medskip

Our assumptions on $\expec{\cdot}$ are the following:
\begin{assumption}\label{A}
  \begin{itemize}
    \item[]
  \item[(A1)] (\textit{Stationarity}). The \textit{shift operators} $\Omega\ni\aa\mapsto\aa(\cdot+z)\in\Omega$, $z\in\Z^d$ preserve the measure $\expec{\cdot}$. (For a bond $\bb=\{x,y\}\in\B^d$ and $z\in\Z^d$ we write
    $\bb+z:=\{x+z,y+z\}$ for the \textit{shift} of $\bb$ by $z$.)
  \item[(A2)] (\textit{Moment condition}). There exists a modulus of
    integrability $\Lambda:[1,\infty)\to[0,\infty)$ such that the distance of neighbors is finite on average in the
    sense that
    \begin{equation*}
      \forall p<\infty\ :\
      \max_{i=1,\ldots,d}\expec{(\dist_{\aa}(0,e_i))^p}^{\frac{1}{p}}\leq\Lambda(p).
    \end{equation*}
  \item[(A3)]  (\textit{Spectral Gap Estimate}). There exists a
    constant $\rho>0$ such that for all $\zeta\in
    L^2(\Omega)$ we have
    \begin{equation*}
      \expec{(\zeta-\expec{\zeta})^2}\leq \frac{1}{\rho}\sum_{\bb\in\B^d}\expec{\left(\frac{\partial\zeta}{\partial\bb}\right)^2},
    \end{equation*}
    where $\frac{\partial\zeta}{\partial\bb}$ denotes the
    \textit{vertical derivative} as defined in
    Definition~\ref{D:vertical} below.
  \end{itemize}
  For technical reasons we need to strengthen (A2):
  \begin{itemize}
  \item[(A2+)] We assume that
    \begin{equation*}
      \forall p<\infty\ :\
      \max_{i=1,\ldots,d}\expec{(\dist_{\aa^{e_i,0}}(0,e_i))^p}^{\frac{1}{p}}\leq\Lambda(p),
    \end{equation*}
    where $\aa^{e_i,0}$ denotes the conductance field obtained by
    ``deleting'' the bond $\{0,e_i\}$ (i.~e. $\aa^{e_i,0}(\bb)=\aa(\bb)$ for all $\bb\neq\{0,e_i\}$ and
$\aa^{e_i,0}(\{0,e_i\})=0$).
  \end{itemize}

\end{assumption}
Let us comment on these properties. A minimal requirement needed for
qualitative stochastic homogenization in the uniformly elliptic case is stationarity and
ergodicity of the ensemble. The basic example for such an ensemble are i.~i.~d.~coefficients which means that $\expec{\cdot}$ is a $\B^d$-fold product of a ``single edge'' probability measure on $[0,1]$. The assumption (A3) is weaker than assuming i.~i.~d., but stronger
than ergodicity. Indeed, in \cite{GNO1} it is shown that any i.~i.~d.
ensemble satisfies (A3) with constant $\rho=1$. Moreover, it is shown
that (A3) can be
seen as a quantification of ergodicity. From the functional analytic point of
view the spectral gap estimate is a Poincar\'e inequality where the
derivative is taken in vertical direction, see below. (The terminology ``vertical'' versus ``horizontal'' is motivated from viewing $\aa\in\Omega$ as a
``height''-function defined on the ``horizontal'' plane $\B^d$). We
recall from \cite{GNO1} the definition of the vertical derivative:
\begin{definition}
  \label{D:vertical}
  For $\zeta\in L^1(\Omega)$ the \textit{vertical derivative} w.~r.~t.
  $\bb\in\B^d$ is given by
  \begin{equation*}
    \frac{\partial \zeta}{\partial \bb}:=\zeta-\expec{\zeta}_\bb,
  \end{equation*}
  where $\expec{\zeta}_{\bb}$ denotes the conditional expectation
  where we condition on $\{\aa(\bb')\}_{\bb'\neq\bb}$. For $\zeta:\Omega\to\R$ sufficiently smooth we denote by
  $\frac{\partial \zeta}{\partial\aa(\bb)}$ the classical partial
  derivative of $\zeta$ w.~r.~t. the coordinate $\aa(\bb)$.
\end{definition}
Property (A2) is a crucial assumption on the  connectedness of the
graph. In particular it implies that almost surely every pair of vertices can be connected
by a path with finite intrinsic length. However, (A2) and (A2+) do not exclude
configurations with coefficients that vanish with non-zero
probability, as it is the case for $\expec{\cdot}_\lambda$ -- the
model considered in the introduction:

\begin{lemma}
  \label{L:regul-perc-model}
  The modified Bernoulli percolation model $\expec{\cdot}_\lambda$
  defined via \eqref{eq:modbernoulli} satisfies Assumption~1 with $\rho=1$.
\end{lemma}
\begin{proof}
  
  Evidently, $\expec{\cdot}_{\lambda}$ can be written as the  (infinite) product of probability measures attached to
  the bonds in $\B^d$. These ``single-bond'' probability measures only
  depend on the direction of the bond. Hence, $\expec{\cdot}_{\lambda}$ is stationary.
  Another consequence of the product structure is that $\expec{\cdot}_{\lambda}$ satisfies
  (A3) with constant $\rho=1$ (see \cite[Lemma~7]{GNO1} for the argument). It
  remains to check (A2+). 
	By stationarity and symmetry we may assume that $e_i=e_d$. Consider the (random) set
  \begin{equation*}
    {\mathcal L}(\aa):=\{\,j\in\Z\,:\,\aa^{e_d,0}(\{je_1,je_1+e_d\})=1\,\}.
  \end{equation*}
  Clearly, each $j\in{\mathcal L(\aa)}$ yields an open path connecting
  $0$ and $e_d$, for instance the ``U-shaped'' path through the sites $0$,
  $je_1$, $je_1+e_d$ and $e_d$. Hence, $\dist_{\aa^{e_d,0}}(0,e_d)\leq
2\dist(0,{\mathcal L(\aa)})+1$ almost surely, where $\dist(0,\mathcal
L(\aa)):=\min_{j\in\mathcal L(\aa)}|j|$. Consequently, it suffices to prove that
  \begin{equation*}
    \expec{(2\dist(0,{\mathcal L(\aa)})+1)^{p}}^{\frac{1}{p}}_\lambda<\infty
  \end{equation*}
  for any $p\geq 1$. Note that due to the definition $\aa^{e_d,0}(\{0,e_d\})=0$ 
	and thus $\dist(0,{\mathcal L(\aa)})\in\N$. Hence,  
  \begin{equation*}
    \expec{(2\dist(0,{\mathcal L(\aa)})+1)^{p}}_\lambda=\sum_{k=1}^\infty
    (2k+1)^{p}\expec{{\boldsymbol 1}(A_k)}_{\lambda},
  \end{equation*}
  where ${\boldsymbol 1}(A_k)$ denotes the set indicator function of $A_k:=\{\,\aa\,:\,\dist(0,{\mathcal L}(\aa))= k\,\}$. Evidently, we have
  \begin{equation*}
    A_k\subset A_k':=\Big\{\,\aa\,:\,\aa(\{je_1,je_1+e_d\})=0\text{ for
      all }|j|=1,\ldots,k-1\,\Big\}.
  \end{equation*}
  From $\expec{{\boldsymbol 1}(A_k')}_{\lambda}=(1-\lambda)^{k-1}$, we
  deduce that
  \begin{equation*}
    \expec{(2\dist(0,{\mathcal L(\aa)})+1)^{p}}_\lambda\leq \sum_{k=1}^\infty
    (2k+1)^{p}(1-\lambda)^{2(k-1)}.
  \end{equation*}
  The sum on the right-hand side converges, since $0<\lambda\leq 1$ by
  assumption.
  This completes the proof.
\end{proof}

\section{Main result}

We are interested in stationary solutions to the corrector equation \eqref{eq:cor-equation-phys}. Note that we tacitly identify the vector $e\in\R^d$ with the translation invariant vector field $e(\bb):=e \cdot(y_{\bb}-x_{\bb})$. For conciseness we write
\begin{align*}
  \mathcal S:=\,\Big\{\,\varphi\,:\,\Omega\times\Z^d\to\R\,\big|\,&\varphi\text{ is measurable and stationary, i.~e. }\varphi(\aa(\cdot+z),x)=\varphi(\aa,x+z)\\
  &\text{for all $x,z\in\Z^d$ and $\expec{\cdot}$-almost every $\aa\in\Omega$}\,\Big\}
\end{align*}
for the space of \textit{stationary random fields}. Thanks to (A1) the
expectation $\expec{\varphi}=\expec{\varphi(\cdot,x)}$ of a stationary
random variable does not depend on $x$. Therefore,
$\|\varphi\|_{L^2(\Omega)}:=\expec{|\varphi|^2}^{\frac{1}{2}}$ defines
a norm on $(\mathcal S,\|\cdot\|_{L^2(\Omega)})$.
\medskip

We are interested in solutions to \eqref{eq:cor-equation-phys} in
$(\mathcal S,\|\cdot\|_{L^2(\Omega})$. Thanks to discreteness, the
operator $\nabla^*(\aa\nabla)$ is bounded and linear on $(\mathcal
S,\|\cdot\|_{L^2(\Omega)})$. However, it is degenerate-elliptic for two-reasons:
\begin{itemize}
\item In general the Poincar\'e inequality  does not hold in $(\mathcal S,\|\cdot\|_{L^2(\Omega)})$.
\item The conductances $\aa$ may vanish with positive probability.
\end{itemize}
Therefore, following \cite{Papanicolaou-Varadhan-79}, we regularize the equation by adding a $0$th order term and consider for $T>0$ the modified corrector equation
\begin{equation}\label{eq:cor-modified}
  \frac{1}{T}\phi_T(x)+\nabla^*\aa(x)(\nabla\phi_T(x)+e )=0\qquad\mbox{
    for all $x\in\Z^d$ and $\aa\in\Omega$}.
\end{equation}
Thanks to the regularization, \eqref{eq:cor-modified} admits (for all $T>0$) a unique solution in $(\mathcal S,\|\cdot\|_{L^2(\Omega)})$ as follows from Riesz' representation theorem. 
\begin{definition}[modified corrector]
  \label{def:1}
  The unique solution $\phi_T\in(\mathcal
  S,\|\cdot\|_{L^2(\Omega)})$ to \eqref{eq:cor-modified} is called the
  modified corrector.
\end{definition}
We think about the modified corrector as an approximation for the stationary corrector and hope to recover a solution to \eqref{eq:cor-equation-phys} in the limit $T\uparrow\infty$. This is possible as soon as we have estimates on (some) moments of $\phi_T$ that are uniform in $T$ --- this is the main result of the paper:
\begin{theo}[Moment bounds for the modified corrector]
 \label{T1}
 Let $d>2$ and $\expec{\cdot}$ satisfy Assumption~\ref{A} for some $\rho$ and
 $\Lambda$. Let
 $\phi_T$ denote the modified corrector as defined in Definition~\ref{def:1}. Then
 for all $T>0$ and $1\leq p<\infty$ we have
 \begin{equation}\label{eq:42}
   \expec{|\phi_T|^{p}}^{\frac{1}{p}}\lesssim 1.
 \end{equation}
 Here $\lesssim$ means $\leq$ up to a constant that
 only depends on $p$, $\Lambda$, $\rho$, and $d$.
\end{theo}
Since the estimate in Theorem \ref{T1} is uniform in $T$ we get as a corollary:
\begin{corollary}
  \label{C1}
  Let $d>2$ and $\expec{\cdot}$ satisfy Assumption~\ref{A}  for some $\rho$ and
  $\Lambda$. Then the corrector equation
  \eqref{eq:cor-equation-phys} has a unique stationary solution $\phi\in(\mathcal S,\|\cdot\|_{L^2(\Omega})$ with $\expec{\phi}=0$. Moreover, we have
  \begin{equation*}
    \expec{|\phi|^p}^{\frac{1}{p}}\lesssim 1
  \end{equation*}
  for all $1\leq p<\infty$. Here $\lesssim$ means $\leq$ up to a  constant that
  only depends on $p$, $\Lambda$, $\rho$ and $d$.
\end{corollary}
As mentioned in the introduction the corrector can be used to
establish invariance principles for random walks in random
environments. Suppose
 that $\expec{\cdot}$ satisfies Assumption~\ref{A} for some $\rho$ and
 $\Lambda$. Then, thanks to Corollary~\ref{C1},  for each coordinate direction
 $e_k$ there exist stationary
 correctors $\phi^k\in(\mathcal
 S,\|\cdot\|_{L^2(\Omega})$ with $\expec{\phi^k}=0$ that solve
 (\ref{eq:cor-equation-phys}) with $e=e_k$. Hence, we can consider the random vector field $\chi=(\chi^1,\ldots,\chi^d):\Omega\times\Z^d\to\R^d$ defined by
\begin{equation*}
  \chi^k(\aa,x):=\phi^k(\aa,x)-\phi^k(\aa,x=0).
\end{equation*}
By construction the map $\Z^d\ni x\mapsto
x+\chi(\aa,x)$ is $\aa$-harmonic, has finite second moments, and is shift covariant (i.e.
$\chi(\aa,x+y)-\chi(\aa,x)=\chi(\aa(\cdot+y),x)$. The field $\chi$ is
precisely the ``corrector'' used e.g. in 
\cite{Kipnis-Varadhan-86, Sidoravicius-Sznitman-04} to introduce
harmonic coordinates for which the
random walk in the random environment is a martingale. In particular, in \cite{Sidoravicius-Sznitman-04} Sidovaricius and Sznitman
use $\chi$ to prove a quenched invariance principle
for the random walk in a random environment.
A key step in
their argument is to show that $\chi$ has sublinear growth, i.e.
\begin{equation}\label{eq:18}
  \lim\limits_{R\to\infty}\max_{x\in B_R(0)}\frac{|\chi(\aa,x)|}{R}=0\qquad\text{for
    $\expec{\cdot}$-almost every }\aa\in\Omega.
\end{equation}
This property has been established for supercritical bond percolation on $\Z^d$ in
dimension $d\geq 4$ in \cite{Sidoravicius-Sznitman-04} and for $d\geq
2$ in \cite{Berger-Biskup-07,Mathieu-Piatnitski-07}.  The moment
bounds established in our work (under the more restrictive Assumption
\ref{A}) are stronger.
Indeed, from Theorem~\ref{T1} we get \eqref{eq:18} in a stronger
form by the following simple argument: For every $\theta\in(0,1)$, $p>\frac{d}{1-\theta}$ and
  $k=1,\ldots,d$ we have
  \begin{eqnarray*}
    R^{\theta-1}\max_{x\in B_R(0)}|\chi^k(\aa,x)|&\leq&
  |\phi^k(\aa,0)|+R^{\theta-1}\max_{x\in B_R(0)}|\phi^k(\aa,x)|\\
  &\leq&
  |\phi^k(\aa,0)|+\left(R^{-d}\sum_{x\in B_R(0)}|\phi^k(\aa,x)|^{\frac{d}{1-\theta}}\right)^{\frac{1-\theta}{d}}.
\end{eqnarray*}
Hence,
since $\phi^k$ is stationary and $p>\frac{d}{1-\theta}$, the maximal function estimate yields
\begin{equation*}
  \expec{\sup_{R\geq 1}\left(\max_{x\in
        B_R(0)}\frac{|\phi^k(x)|}{R^{1-\theta}}\right)^p}\leq C\expec{|\phi^k|^p}.
\end{equation*}
With the moment bounds of Corollary~\ref{C1} we
get for $\expec{\cdot}$ satisfying Assumption~\ref{A}:
\begin{equation}\label{eq:21}
\forall\theta\in(0,1)\,:\qquad  \lim\limits_{R\to\infty}\max_{x\in
    B_R(0)}\frac{|\chi(\aa,x)|}{R^{1-\theta}}=0\qquad\text{$\expec{\cdot}$-almost surely.}
\end{equation}

\subsection{Outline and Proof of Theorem~\ref{T1}}\label{S3}
The proof of Theorem~\ref{T1} is inspired by the approach in
\cite{GO1} where uniformly elliptic conductances are treated.  The starting point of our argument is the following $p$-version of the
{\em Spectral Gap Estimate} (A3), which we recall from \cite[Lemma~2]{GNO1}:
\begin{lemma}[p-version of (SG)]
  \label{lem:SGp}
  Let $\expec{\cdot}$ satisfy (A3) with constant $\rho>0$. Then for $p\in\mathbb{N}$ and all $\zeta\in
  L^{2p}(\Omega)$ with $\expec{\zeta}=0$ we have
  \begin{equation*}
    \expec{\zeta^{2p}}\lesssim \expec{\left(\sum_{\bb\in\B^d}
        \left(\frac{\partial\zeta}{\partial\bb}\right)^2\right)^p},
  \end{equation*}
  where $\lesssim$ means $\leq$ up to a  constant that
  only depends on $p$, $\rho$ and $d$.
\end{lemma}
Applied to $\zeta=\phi_T(x=0)$, this estimate yields a bound
on stochastic moments of $\phi_T$ in terms of the vertical derivatives
$\frac{\partial\phi_T(x=0)}{\partial\bb}$, $\bb\in\B^d$ (see
Definition~\ref{D:vertical}). Heuristically, we expect the vertical derivative
$\frac{\partial\phi_T(x=0)}{\partial\bb}$ to behave as the classical partial
derivative $\frac{\partial\phi_T(x=0)}{\partial\aa(\bb)}$. As we shall see, the latter
admits the Green's function representation
\begin{equation}\label{eq:43}
  \frac{\partial\phi_T(x=0)}{\partial\aa(\bb)}=-\nabla G_T(\aa,\bb,0)(\nabla\phi_T(\bb)+e(\bb)).
\end{equation}
Here $G_T$ denotes the Green's function associated with
$(\frac{1}{T}+\nabla^*\aa\nabla)$ and is defined as follows:
\begin{definition}
  \label{D:G}
  For $T>0$ the Green's function
  $G_T:\Omega\times\Z^d\times\Z^d\to\R$ is defined as follows: For
  each $\aa\in\Omega$ and $y\in\Z^d$ the function $x\mapsto
  G_T(\aa,x,y)$ is the unique  solution in $\ell^2(\Z^d)$ to 
  \begin{equation}\label{eq:D:G}
    \frac{1}{T} G_T(\aa,\cdot,y)+\nabla^*\aa\nabla G_T(\aa,\cdot,y)=\delta(\cdot-y).
  \end{equation}
\end{definition}
For uniformly elliptic conductances we have
$\frac{\partial\phi_T(x=0)}{\partial\bb}\sim\frac{\partial\phi_T(x=0)}{\partial\aa(\bb)}$
up to a constant that only depends on the ratio of ellipticity. In the
case of degenerate ellipticity this is no longer true. However,
the discrepancy between the vertical and classical partial derivative
of $\phi_T$ can be quantified in terms of  weights defined as follows:
We introduce the weight function $\omega:\Omega\times\B^d\to[0,\infty]$ as
\begin{equation}\label{D:omega}
    \omega(\aa,\bb):=(\dist_{\aa}(x_{\bb},y_{\bb}))^{d+2}\qquad\qquad (\aa\in\Omega,\ \bb=\{x_{\bb},y_{\bb}\}\in\B^d).
\end{equation}
For $\bb\in\B^d$ and $\aa\in\Omega$ we denote by
$\aa^{\bb,0}$ the conductance field obtained by ``deleting'' the bond
$\bb$ (i.~e. $\aa^{\bb,0}(\bb')=\aa(\bb')$ for all $\bb'\neq\bb$ and
$\aa^{\bb,0}(\bb)=0$), and introduce the modified weight $\omega_0$ as
\begin{equation}\label{D:tildeomega}
  \omega_0(\aa,\bb):=\omega(\aa^{\bb,0},\bb).
\end{equation}

\begin{lemma}
  \label{lem:2}
  Assume that $\expec{\cdot}$ satisfies (A1) and (A2+). For $T>0$ let $\phi_T$ denote the modified corrector. Then for all $\bb\in\B^d$ we have
  \begin{equation*}
    \left|\frac{\partial\phi_T(x=0)}{\partial\bb}\right|\lesssim \omega_0^{2}(\bb)\left|\nabla G_T(\bb,0)\right|\,\left|\nabla\phi_T(\bb)+e(\bb)\right|.
  \end{equation*}
  Here  $\lesssim$ means $\leq$ up to a  constant that only depends on $d$.
\end{lemma}
To benefit from \eqref{eq:43} (in the form of Lemma~\ref{lem:2}) we require an {\em estimate on the gradient of the
  Green's function}. As it is well known, the constant coefficient Green's function
$G_T^0(x):=G_T(\aa={\boldsymbol 1},x,0)$ (which is associated with the modified
Laplacian $\frac{1}{T}+\nabla^*\nabla$) satisfies the pointwise estimate 
\begin{equation}\label{eq:5}
  \forall \bb:=\{x,x+e_i\}\,:\qquad |\nabla G_T^0(\bb)|\lesssim (1+|x|)^{1-d}\qquad\text{uniformly in $T>0$}.
\end{equation}
We require an estimate that captures the same decay in $x$. It is
known from the continuum, uniformly elliptic case, that such an
estimate cannot hold pointwise in $x$ and at the same time pointwise in $\aa$.  In \cite[Lemma~2.9]{GO1}, for  uniformly elliptic conductances, a spatially averaged version of \eqref{eq:5} is established, where the averages are taken over dyadic annuli. The constant in this estimate depends on the conductances only through their contrast of ellipticity. In the degenerate elliptic case, the ellipticity contrast is infinite. In order to keep the optimal decay in $x$, we need to allow the constant in the estimate to depend on $\aa$. For $x_0\in\Z^d$, $R>1$ and $1\leq q<\infty$ consider the spatial average of the weight $\omega$ (cf. \eqref{D:omega}) 
\begin{equation}
  \label{eq:C}
  C(\aa,Q_R(x_0),q):=\left(\frac{1}{|Q_R(x_0)|}\sum_{\bb\in Q_R(x_0)}\omega^q(\aa,\bb)\right)^{\frac{1}{q}}.
\end{equation}
We shall prove the following estimate:
\begin{prop}\label{P1}
  For $R_0>1$ and $k\in\N_0$ consider
  \begin{equation*}
    A_k:=\left\{
      \begin{aligned}
        &Q_{R_0}(0)&&k=0,\\
        &Q_{2^{k}R_0}(0)\setminus Q_{2^{k-1}R_0}(0)&&k\geq 1.
      \end{aligned}
    \right.
  \end{equation*}
  Then for all $\frac{2d}{d+2}<p<2$ we have
  \begin{equation*}
    \left(\frac{1}{|A_k|}\sum_{\bb\in A_k}|\nabla
      G_T(\aa,\bb,0)|^p\right)^{\frac{1}{p}}\lesssim C(\aa)\,2^{k(1-d)},
  \end{equation*}
  where $\lesssim$ means $\leq$ up to a  constant that
  only depends on $R_0$, $d$ and $p$, and
  \begin{equation}    \label{eq:CP1}
    C(\aa):=C^{\frac{\beta}{2}}(\aa,Q_{2^{k+1}R_0}(0),\tfrac{p}{2-p})
  \end{equation}
  with $\beta:=2\frac{p^*-1}{p^*-2}+p^*$ and $p^*:=\frac{dp}{d-p}$. 
\end{prop}
The precise form of the constant $C$ in \eqref{eq:CP1} is not crucial. In fact, in the random setting, when $\Omega$ is equipped with a probability
measure satisfying (A1) and (A2), we may view  $C$ as a random variable with controlled finite moments:
\begin{rem}
  \label{lem:3}
  Let $\expec{\cdot}$ satisfy Assumption (A1). Then the spatial
  average introduced in \eqref{eq:C} satisfies
  \begin{eqnarray*}
    \expec{ C^{q}(\aa, Q_{R}(x_0),q')}=\expec{\left(\frac{1}{|Q_{R}(x_0)|}
        \sum_{\bb\in Q_{R}(x_0)}\omega^{q'}(\aa,\bb)\right)^{\frac{q}{q'}}} \leq
    \begin{cases}
      \expec{\omega^{q'}}^{\frac{q}{q'}}&\text{if } q'\geq q,\\
      \expec{\omega^q}&\text{if  }q'<q,
    \end{cases}
  \end{eqnarray*}
  as can be seen by appealing to Jensen's inequality and stationarity. Moreover, if $\expec{\cdot}$ additionally fulfills (A2), then $C$ defined in \eqref{eq:CP1} satisfies
  \begin{equation*}
    \forall m\in\N\,:\,\expec{C^m}^{\frac{1}{m}}\lesssim 1,
  \end{equation*}
  where $\lesssim$ means $\leq$ up to a  constant that
  only depends on $m$, $p$, $\Lambda$ and $d$.
\end{rem}
\medskip

The proof of Proposition~\ref{P1} relies on arguments from  elliptic regularity theory, which in the uniformly elliptic case are standard.  
They typically involve the pointwise inequality
\begin{equation}\label{eq:40}
  \lambda_0|\nabla u(\bb)|^2\leq\nabla u(\bb)\,\aa(\bb)\nabla u(\bb),\qquad(\bb\in\B^d),
\end{equation}
where $\lambda_0>0$ denotes the constant of ellipticity. In the
degenerate case, the conductances $\aa$ may vanish on a non-negligible
set of bonds and \eqref{eq:40} breaks down. As a replacement we establish estimates
which provide a weighted, integrated version of \eqref{eq:40}:

\begin{lemma}
  \label{lem:coercivity-phys}
  Let  $p>d+1$. For any function $u:\Z^d\to\R$ and all $\aa\in\Omega$ we have (with the convention $\frac{1}{\infty}=0$)
  \begin{equation}\label{eq:coercivity}
    \sum_{\bb\in\B^d}|\nabla
    u(\bb)|^2\dist_{\aa}^{-p}(x_{\bb},y_{\bb})\leq C(p,d)\sum_{\bb\in\B^d}\aa(\bb)|\nabla
    u(\bb)|^2,
  \end{equation}
  where $C(p,d):=\sum_{x\in\Z^d}(|x|+1)^{1-p}$ and the inequality holds whenever the sums converge.
\end{lemma}
While Lemma~\ref{lem:coercivity-phys} is purely deterministic, we also need the following statistically averaged version:
\begin{lemma}
  \label{lem:coercivity-prob}
  Let $\expec{\cdot}$ be stationary, cf. (A1), and $p>d+1$. Then for any
  stationary random field $u$ and any bond $\bb\in\B^d$ we have (with the convention $\frac{1}{\infty}=0$)
  \begin{equation*}
    \expec{|\nabla u(\bb)|^2\dist_{\aa}^{-p}(x_{\bb},y_{\bb})}\leq
    C(p,d) \sum_{\bb'=\{0,e_i\}\atop i=1\ldots
      d}\expec{\aa(\bb')|\nabla u(\bb')|^2},
  \end{equation*}
  where $C(p,d):=\sum_{k=0}^\infty2^{k(1-p)}|B_{2^{k+1}}(0)|<\infty$.
\end{lemma}

A last ingredient required for the proof of Theorem~\ref{T1} is a {\em
  Caccioppoli inequality in  probability} that yields a gain of
stochastic integrability and helps to treat the $\nabla\phi_T$-term on
the right-hand side in \eqref{eq:43}. In the uniformly elliptic case, i.~e.~when $0<\lambda_0\leq\aa\leq1$, the Caccioppoli inequality
\begin{equation}\label{eq:30}
  \expec{|\nabla\phi_T|^{2p+2}}^{\frac{1}{2p+2}}\lesssim\expec{\phi_T^{2p}}^{\frac{1}{2p}\frac{p}{p+1}}
\end{equation}
holds for any integer exponents $p$ (see  \cite[Lemma~2.7]{GO1}). The inequality follows from combining
the elementary discrete
inequality
\begin{equation}\label{eq:discrete}
  |\nabla u(\bb)|=|u(y_{\bb})-u(x_{\bb})|\leq|u(y_{\bb})|+|u(x_{\bb})|,
\end{equation}
with the estimate
\begin{equation}\label{eq:29}
  \expec{\phi_T^{2p}|\nabla\phi_T|^2}\lesssim\frac{1}{\lambda_0}\expec{\phi^{2p}_T|\nabla\phi_T|}.
\end{equation}
The latter is obtained by testing the modified corrector equation
\eqref{eq:cor-modified} with
$\phi_T^{2p+1}$ and uses the uniform ellipticity of $\aa$. In the
degenerate elliptic case \eqref{eq:29} is not true any longer.  However, by appealing to Lemma
\ref{lem:coercivity-prob} the following weaker version of \eqref{eq:30} survives:
\begin{equation}\label{eq:31}
  \expec{|\nabla\phi_T|^{(2p+2)\theta}}^{\frac{1}{(2p+2)\theta}}\lesssim\expec{\phi^{2p}_T}^{\frac{1}{2p}\frac{p}{p+1}}
\end{equation}
for any factor $0<\theta<1$. Hence, we only gain an increase of
integrability by exponents strictly smaller two. As a matter of fact, in the proof of our main result we only need the estimate in
the following form:
\begin{lemma}[Caccioppoli estimate in probability]
  \label{L:CEP}
  Let $\expec{\cdot}$ satisfy (A1) and (A2). 
  Let $\phi_T$ denote the corrector associated with $e \in\R^d$,
  $|e |=1$, $T>0$.
  For every even integer $p$ we have
  \begin{equation}
    \label{eq:cacciopprob}
    \expec{|\nabla\phi_T|^{2p+1}}^{\frac{1}{2p+1}}\lesssim\expec{\phi^{2p}_T}^{\frac{1}{2p}\frac{p}{p+1}},
  \end{equation}
  where $\lesssim$ means $\leq$ up to a  constant that
  only depends on $p$, $\Lambda$ and $d$.
\end{lemma}

Now we are ready to prove our main result:
\begin{proof}[Proof of Theorem~\ref{T1}]
  It suffices to consider exponents $p\in 2\N$ that are larger than a
  threshold only depending on $d$ -- the threshold is determined by
  \eqref{eq:23} below.

  Further, we only need to prove 
  \begin{equation}\label{eq:6}
    \expec{\phi_T^{2p}}^{\frac{1}{2p}}\lesssim
    \max_{\bb'=\{0,e_i\}\atop i=1,\ldots,d}\expec{|\nabla\phi_T(\bb')|^{2p+1}}^{\frac{1}{2p+1}} + 1.
  \end{equation}
  Indeed, in combination with the Caccioppoli estimate in
  probability, cf. Lemma~\ref{L:CEP}, estimate \eqref{eq:6}  yields $\expec{\phi_T^{2p}}^{\frac{1}{2p}}\lesssim\expec{\phi_T^{2p}}^{\frac{1}{2p}\frac{p}{p+1}}
  + 1$. Since $\frac{p}{p+1}<1$ the first term can be absorbed and the
  desired estimate follows.

  We prove \eqref{eq:6}. For reasons that will become clear at the end of
  the argument we fix an exponent $\frac{2d}{d+2}<q<2$ such that
   \begin{equation}\label{eq:23}
     d(\frac{1}{q}+\frac{1}{2p}-1)+1<0.
   \end{equation}
   This is always possible for $p\gg 1$ and $0<2-q\ll 1$, since
   \begin{equation*}
     \lim_{q\uparrow 2,p\uparrow\infty}
     d(\frac{1}{q}+\frac{1}{2p}-1)=-\frac{d}{2}<-1\qquad\text{for }d>2.
   \end{equation*}
   Our argument for \eqref{eq:6} starts with the $p$-version of the  spectral gap
   estimate, see Lemma~\ref{lem:SGp}, that we combine with
   Lemma~\ref{lem:2}:
   \begin{eqnarray*}
     \expec{\phi_T^{2p}}^{\frac{1}{p}}&=&\expec{\phi_T^{2p}(x=0)}^{\frac{1}{p}}\lesssim\expec{\left(\sum_{\bb\in\B^d}\left(\frac{\partial\phi_T(x=0)}{\partial\bb}\right)^2\right)^p}^{\frac{1}{p}}\\
     &\lesssim&\expec{\left(\sum_{\bb\in\B^d}(\nabla G_T(\bb,0))^2(\nabla\phi_T(\bb)+e (\bb))^2\omega_0^{4}(\bb)\right)^p}^{\frac{1}{p}}.
   \end{eqnarray*}
   Now we wish to benefit from the decay estimate for $\nabla G_T$ in Proposition~\ref{P1},
   and therefore decompose $\B^d$ into dyadic annuli: Let the dyadic annuli $A_k$, $k\in\N_0$ be defined as in
   Proposition~\ref{P1} with initial radius $R_0=2$. Note that $\B^d$ can be written as the
   disjoint union of $A_0,A_1,A_2,\ldots$ .
   With the triangle inequality w.~r.~t. $\expec{(\cdot)^p}^{\frac{1}{p}}$ and
   H\"older's inequality in $\bb$-space with exponents
   $(\frac{p}{p-1},p)$ we get
   \begin{align}\label{eq:25}
     \begin{split}
       \expec{\phi_T^{2p}}^{\frac{1}{p}}
       &\lesssim\ \sum_{k\in\N_0}\expec{\left(\sum_{\bb\in A_k}(\nabla
           G_T(\bb,0))^2(\nabla\phi_T(\bb)+e (\bb))^2\omega_0^{4}(\bb)\right)^p}^{\frac{1}{p}}\\
       &\lesssim\  \sum_{k\in\N_0}\expec{ \left(\sum_{\bb\in A_k}|\nabla
           G_T(\bb,0)|^{\frac{2p}{p-1}}\right)^{p-1}\left(\sum_{\bb\in A_k}
           (\nabla\phi_T(\bb)+e (\bb))^{2p}\omega_0^{4p}(\bb)\right)}^{\frac{1}{p}}.
     \end{split}
   \end{align}
   Because $\frac{2d}{d+2}<q<2<\frac{2p}{p-1}$, the discrete
   $\ell^q$-$\ell^{\frac{2p}{p-1}}$-estimate combined with the decay
   estimate of Proposition~\ref{P1} yields
   \begin{eqnarray}\label{eq:24}
     \left(\sum_{\bb\in A_k}|\nabla
       G_T(\bb,0)|^{\frac{2p}{p-1}}\right)^{p-1}
     &\leq&
     \left(\sum_{\bb\in A_k}|\nabla
       G_T(\bb,0)|^{q}\right)^{\frac{2p}{q}}\leq C 2^{k(2p(1-(1-\frac{1}{q})d))}.
   \end{eqnarray}
   Here and below, $C$ denotes a generic, non-negative
   random variable with the property that $\expec{C^m}\lesssim 1$ for
   all $m<\infty$, where $\lesssim$ means $\leq$ up to a  constant that
   only depends on $m$, $p$, $q$, $\Lambda$ and $d$.  Combining \eqref{eq:25} and
   \eqref{eq:24} yields
   \begin{equation}\label{eq:26}
     \expec{\phi_T^{2p}}^{\frac{1}{p}}
     \lesssim
     \sum_{k\in\N_0}2^{2k(1-(1-\frac{1}{q})d)}\left(
       \sum_{\bb\in A_k}
       \expec{C\ (\nabla\phi_T(\bb)+e (\bb))^{2p}\ \omega_0^{4p}(\bb)} \right)^{\frac{1}{p}}.
   \end{equation}
   Next we apply a triple H\"older inequality in probability with
   exponents $(\theta,\theta',\theta')$, where we choose
   $\theta=\frac{2p+1}{2p}$ (so that $2p\theta=2p+1$). We have
   \begin{eqnarray*}
     \expec{C\ (\nabla\phi_T(\bb)+e (\bb))^{2p}\ \omega_0^{4p}(\bb)}
     \leq
     \expec{(\nabla\phi_T(\bb)+e (\bb))^{2p+1}}^{\frac{2p}{2p+1}}   
     \expec{C^{\theta'}}^{\frac{1}{\theta'}}
     \expec{\omega_0^{4p\theta'}(\bb)}^{\frac{1}{\theta'}}.
   \end{eqnarray*}
   The first term is estimated by stationarity of $\nabla\phi_T$ and the assumption $|e |=1$ as
   \begin{equation*}
     \expec{(\nabla\phi_T(\bb)+e (\bb))^{2p+1}}^{\frac{2p}{2p+1}}\lesssim
     \max_{\bb'=\{0,e_i\}\atop
       i=1,\ldots,d}\expec{|\nabla\phi_T(\bb')|^{2p+1}}^{\frac{2p}{2p+1}}
     + 1.
   \end{equation*}
   For the second term we have $\expec{C^{\theta'}}^{\frac{1}{\theta'}}\expec{\omega_0^{4p\theta'}(\bb)}^{\frac{1}{\theta'}}\lesssim
   1$ due to (A2+), so that we obtain
   \begin{equation}\label{eq:27}
      \expec{C\ (\nabla\phi_T(\bb)+e (\bb))^{2p}\ \omega_0^{4p}(\bb)}
     \lesssim \max_{\bb'=\{0,e_i\}\atop
       i=1,\ldots,d}\expec{|\nabla\phi_T(\bb')|^{2p+1}}^{\frac{2p}{2p+1}}
     + 1.
   \end{equation}
   Combined with \eqref{eq:26} we get
   \begin{eqnarray*}
     \expec{\phi_T^{2p}}^{\frac{1}{p}}
     &\lesssim& \left(\max_{\bb'=\{0,e_i\}\atop
       i=1,\ldots,d}\expec{|\nabla\phi_T(\bb')|^{2p+1}}^{\frac{2}{2p+1}}+1\right)\
     \times\ 
     \sum_{k\in\N_0}2^{2k(1-(1-\frac{1}{q})d)}|A_k|^{\frac{1}{p}}\\
     &\lesssim& \left(\max_{\bb'=\{0,e_i\}\atop
       i=1,\ldots,d}\expec{|\nabla\phi_T(\bb')|^{2p+1}}^{\frac{2}{2p+1}}+1\right).
   \end{eqnarray*}
   In the last line we used that
   \begin{equation*}
     \sum_{k\in\N_0}2^{2k(1-(1-\frac{1}{q})d)}|A_k|^{\frac{1}{p}}\lesssim\sum_{k\in\N_0}2^{2k(1-(1-\frac{1}{2p}-\frac{1}{q})d)}\lesssim 1,
   \end{equation*}
   which holds since the exponent is negative, cf. \eqref{eq:23}.
   This proves \eqref{eq:6}.
\end{proof}

\section{Proofs of the auxiliary lemmas}\label{S:proofs}
\subsection{Proof of Lemma~\ref{lem:2}}
The argument for Lemma~\ref{lem:2} is split into three lemmas.

\begin{lemma}\label{L:ODE}
  Let $\bb\in\B^d$ be fixed.  For $T>0$ let $\phi_T$ and $G_T$ denote the modified corrector
  and the Green's function, respectively. Then 
  \begin{eqnarray}
    \label{eq:ODE:0}
    \frac{\partial\phi_T(x=0)}{\partial \aa(\bb)}&=&-\nabla G_T(\bb,0)(\nabla\phi_T(\bb)+e (\bb)),\\
    \label{eq:ODE:1}
    \frac{\partial}{\partial \aa(\bb)}\frac{\partial\phi_T(x=0)}{\partial \aa(\bb)}
    &=&-2\nabla\nabla G_T(\bb,\bb)\frac{\partial\phi_T(x=0)}{\partial \aa(\bb)},\\
    \label{eq:ODE:2}
    \frac{\partial}{\partial\aa(b)}\nabla\nabla G_T(\bb,\bb)
    &=&-\left(\nabla\nabla G_T(\bb,\bb)\right)^2.
  \end{eqnarray}
  Moreover, $\nabla\nabla G_T(\bb,\bb)$ and $1-\aa(\bb)\nabla\nabla
  G_T(\bb,\bb)$ are strictly positive.
\end{lemma}

\begin{proof}[Proof of Lemma~\ref{L:ODE}]
  For simplicity we write $\phi$ and $G$ instead of $\phi_T$ and
$G_T$.
\medskip

\step 1 Argument for \eqref{eq:ODE:2}.

We first claim that
\begin{subequations}
  \begin{align}
    \label{eq:rel3}
    &\frac{\partial}{\partial \aa(\bb)}G(x,y)=-\nabla G(\bb,y)\nabla G(\bb,x),\\
    \label{eq:rel4}
    &\frac{\partial}{\partial \aa(\bb)}\nabla G(x,\bb)=-\nabla\nabla G(\bb,\bb)\nabla G(\bb,x).
  \end{align}
\end{subequations}
Indeed, since $\nabla$ and $\frac{\partial}{\partial\aa(\bb)}$
commute, an application of $\frac{\partial}{\partial \aa(\bb)}$ to
\eqref{eq:D:G} yields
\begin{equation}\label{eq:ODE:step-1:1}
  \left(\frac{1}{T}+\nabla^*\aa\nabla\right)\frac{\partial G(\cdot,y)}{\partial\aa(\bb)}=-\nabla^*\frac{\partial \aa(\cdot)}{\partial \aa(\bb)}\nabla G(\cdot,y). 
\end{equation}
We test this identity with $G(\cdot,x)$:
\begin{eqnarray}
  \label{eq:ODE:step-1:2}
  \frac{\partial G(x,y)}{\partial
      \aa(\bb)}&=&\sum_{y'\in\Z^d}\frac{\partial G(y',y)}{\partial
      \aa(\bb)} \delta(x-y')\\\nonumber
  &\stackrel{\eqref{eq:D:G}}{=}&\sum_{y'\in\Z^d}\frac{\partial G(y',y)}{\partial
      \aa(\bb)} \,\left(\frac{1}{T}+\nabla^*\aa\nabla\right)G(y',x)\\\nonumber
  &\stackrel{\eqref{int-by-parts}}{=}&\sum_{y'\in\Z^d}G(y',x)\,\left(\frac{1}{T}+\nabla^*\aa\nabla\right)\frac{\partial
  G(y',y)}{\partial \aa(\bb)}\\\nonumber
  &\stackrel{\eqref{eq:ODE:step-1:1},\eqref{int-by-parts}}{=}&-\sum_{\bb'\in\B^d}\frac{\partial \aa(\bb')}{\partial
    \aa(\bb)}\nabla G(\bb',y)\,\nabla G(\bb',x).
\end{eqnarray}
Since $\frac{\partial \aa(\bb')}{\partial
    \aa(\bb)}$ is equal to $1$ if $\bb'=\bb$ and $0$ else, the sum on
  the right-hand side reduces to $\nabla G(\bb,y)\nabla G(\bb,x)$ and we get
  \eqref{eq:rel3}. An application of $\nabla$ to \eqref{eq:rel3} yields
  \eqref{eq:rel4}, and an application of $\nabla$ to \eqref{eq:rel4}
  finally yields \eqref{eq:ODE:2}.
\medskip

\step 2 Argument for \eqref{eq:ODE:0} and \eqref{eq:ODE:1}.

We apply $\frac{\partial}{\partial\aa(\bb)}$ to the
  modified corrector equation \eqref{eq:cor-modified}:
  \begin{equation}\label{eq:ODE:step1:1}
    \frac{1}{T}\frac{\partial\phi}{\partial\aa(\bb)}+\nabla^*\aa\nabla\frac{\partial\phi}{\partial\aa(\bb)}=-\nabla^*\frac{\partial\aa(\cdot)}{\partial\aa(\bb)}(\nabla\phi+e(\bb) ).
  \end{equation}
  As in \eqref{eq:ODE:step-1:2} testing with $G(\cdot,x)$ yields
  \begin{equation}\label{eq:41}
    \frac{\partial\phi(x)}{\partial\aa(\bb)}=-(\nabla\phi(\bb)+e (\bb))\nabla G(\bb,x),
  \end{equation}
  and \eqref{eq:ODE:0} follows. By
  applying $\frac{\partial}{\partial\aa(\bb)}$ and $\nabla$ to
  \eqref{eq:41} we obtain the two identities
  \begin{eqnarray*}
    \frac{\partial}{\partial\aa(\bb)}\frac{\partial\phi(x)}{\partial\aa(\bb)}&=&-
    \frac{\partial(\nabla\phi(\bb)+e (\bb))}{\partial\aa(\bb)}\nabla
    G(\bb,x)-(\nabla\phi(\bb)+e (\bb))\frac{\partial\nabla
      G(\bb,x)}{\partial\aa(\bb)},\\
    \nabla\frac{\partial\phi(\bb)}{\partial\aa(\bb)}&=&-(\nabla\phi(\bb)+e (\bb))\nabla\nabla G(\bb,\bb).
  \end{eqnarray*}
  By combining the first with the second identity, \eqref{eq:rel4} and \eqref{eq:41}
  we get
  \begin{eqnarray*}
    \frac{\partial}{\partial\aa(\bb)}\frac{\partial\phi(x)}{\partial\aa(\bb)}&=&
    2(\nabla\phi(\bb)+e (\bb))\nabla\nabla G(\bb,\bb)\nabla
    G(\bb,x)\\
    &=&-2\frac{\partial\phi(x)}{\partial\aa(\bb)}\nabla\nabla G(\bb,\bb),
  \end{eqnarray*}
  and thus \eqref{eq:ODE:1}.
  \medskip

  \step 3 Positivity of $\nabla\nabla G(\bb,\bb)$ and
  $1-\aa(\bb)\nabla\nabla G(\bb,\bb)$.

  Let $\bb=(x_{\bb},y_{\bb})\in\B^d$ be fixed. An application of $\nabla$
  (w.~r.~t. the $y$-component) to
  \eqref{eq:D:G} yields 
  \begin{equation*}
    (\frac{1}{T}+\nabla^*\aa\nabla)\nabla G(\cdot,\bb)=\delta(\cdot-y_{\bb})-\delta(\cdot-x_{\bb}).
  \end{equation*}
  We test this equation with $\nabla G(\cdot,\bb)$ and get
  \begin{equation}\label{eq:46}
    \frac{1}{T}\sum_{x\in\Z^d}\left(\nabla
      G(x,\bb)\right)^2+\sum_{\bb'\in\B^d}\aa(\bb')\left(\nabla\nabla
      G(\bb',\bb)\right)^2=\nabla\nabla G(\bb,\bb).
  \end{equation}
  This identity implies that $\nabla\nabla G(\bb,\bb)$ and
  $1-\aa(\bb)\nabla\nabla G(\bb,\bb)$ are strictly positive. Indeed,
  $\nabla\nabla G(\bb,\bb)$ must be strictly positive, since otherwise $\sum_{x\in\Z^d}|\nabla G(x,\bb)|^2=0$ and thus $G(\cdot,\bb)=0$ in
  contradiction to \eqref{eq:D:G}. The strict positivity of
  $1-\aa(\bb)\nabla\nabla G(\bb,\bb)$ follows from the strict positivity of
  $\nabla\nabla G(\bb,\bb)-\aa(\bb)\left(\nabla\nabla
    G(\bb,\bb)\right)^2$. The latter can be seen by the following argument:
  \begin{eqnarray*}
    \lefteqn{\nabla\nabla G(\bb,\bb)-\aa(\bb)\left(\nabla\nabla
    G(\bb,\bb)\right)^2}&&\\
    &=&\left(\nabla\nabla G(\bb,\bb)-\frac{1}{T}\sum_{x\in\Z^d}\left(\nabla
      G(x,\bb)\right)^2-\sum_{\bb'\in\B^d}\aa(\bb')\left(\nabla\nabla
      G(\bb',\bb)\right)^2\right)\\
  &&+\frac{1}{T}\sum_{x\in\Z^d}\left(\nabla
      G(x,\bb)\right)^2+\sum_{\bb'\neq\bb}\aa(\bb')\left(\nabla\nabla
      G(\bb',\bb)\right)^2\\
    &\stackrel{\eqref{eq:46}}{\geq}&\frac{1}{T}\sum_{x\in\Z^d}\left(\nabla
      G(x,\bb)\right)^2>0.
  \end{eqnarray*}
\end{proof}

The next lemma establishes a (quantitative) link between the vertical and classical
partial derivative of $\phi_T$. 
\newcommand{\osc}{\mathop{\operatorname{osc}}}
\begin{lemma}\label{L:osc}
  Let $\bb\in\B^d$ be fixed.  For $T>0$ let $\phi_T$ and $G_T$ denote the modified corrector
  and the Green's function. Then 
  \begin{equation}
    \label{eq:osc:1}
    \left|\frac{\partial\phi_T(x=0)}{\partial\bb}\right|\leq\left(1+\frac{\aa(\bb)}{1-\aa(\bb)\nabla\nabla
        G_T(\bb,\bb)}\right)\left|\frac{\partial\phi_T(x=0)}{\partial\aa(\bb)}\right|.
  \end{equation}
\end{lemma}
\begin{proof}[Proof of Lemma~\ref{L:osc}]
  Fix $\aa\in\Omega$ and $\bb\in\B^d$. Set $a_0:=\aa(\bb)$. We shall
  use the following shorthand notation
  \begin{equation}\label{eq:44}
    \varphi(a):=\frac{\partial\phi_T(\aa^{\bb,a},x=0)}{\partial\aa(\bb)},\qquad
    g(a):=\nabla\nabla G_T(\aa^{\bb,a},\bb,\bb),\qquad (a\in[0,1]),
  \end{equation}
  where $\aa^{\bb,a}$ denotes the coefficient field obtained from
  $\aa$ by setting $\aa^{\bb,a}(\bb')=a$ if $\bb'=\bb$ and
  $\aa^{\bb,a}(\bb'):=\aa(\bb')$ else. With that notation \eqref{eq:ODE:1} and \eqref{eq:ODE:2} turn into
  \begin{eqnarray}
    \label{eq:L:osc:1a}
    \varphi'=-2g\varphi,\\
    \label{eq:L:osc:1b}
    g'=-g^2.
  \end{eqnarray}
  Since we have $\left|\frac{\partial\phi_T(x=0)}{\partial\bb}\right|\leq
  \int_0^1|\varphi(a)|\,da$, it suffices to show
  \begin{equation}\label{eq:L:osc:2}
    \int_0^1|\varphi(a)|\,da\leq\left(1+\frac{a_0}{1-a_0g(a_0)}\right)|\varphi(a_0)|.
  \end{equation}
  The positivity of $g$ and \eqref{eq:L:osc:1a} imply that $\varphi$
  is either strictly positive, strictly
  negative or that it vanishes identically. In the latter case, the claim is
  trivial. In the other cases we have 
  \begin{equation*}
    \varphi(a)=\exp(h(a))\varphi(a_0),\qquad\mbox{where }h(a):=\ln\frac{\varphi(a)}{\varphi(a_0)},
  \end{equation*}
  and \eqref{eq:L:osc:2} reduces to the inequality
  \begin{equation}
    \label{eq:19}
    \int_0^1\exp(h(a))\,da\leq 1+ \frac{a_0}{1-a_0g(a_0)}.
  \end{equation}
  From \eqref{eq:L:osc:1a} we learn that $h'=-2g$. Since $g>0$, $h$
  is decreasing. Combined with the identity
  $h(a_0)=0$ we get
  \begin{equation}\label{eq:20}
    h(a)\leq
    \left\{\begin{aligned}
        &2\int_a^{a_0}g(a')\,da'&\qquad&\mbox{for }a\in[0,a_0),\\
        &0&&\mbox{for }a\in[a_0,1].
    \end{aligned}\right.
  \end{equation}
  On the other hand, we learn from integrating \eqref{eq:L:osc:1b}
  that $g(a')=\frac{g(a_0)}{1+(a'-a_0)g(a_0)}$. Hence, for $a<a_0$ the
  right-hand side in \eqref{eq:20} turns into
  \begin{equation*}
    2\int_a^{a_0}g(a')\,da'=-2\ln(1+(a-a_0)g(a_0)),
  \end{equation*}
  which in combination with \eqref{eq:20} yields \eqref{eq:19}.
\end{proof}

Lemma~\ref{lem:2} is a direct consequence of \eqref{eq:osc:1},
\eqref{eq:ODE:0} and the following estimate:

\begin{lemma}\label{L:bound}
  Let $G_T$ denote the Green's function. Assume that (A1) is
  satisfied. Then for all $T>0$, $\aa\in\Omega$ and $\bb\in\B^d$ we have
  \begin{eqnarray}
    \label{eq:bound}
    1+\frac{\aa(b)}{1-\aa(b)\nabla\nabla G_T(\bb,\bb)}\lesssim \omega_0^2(\aa,\bb),
  \end{eqnarray}
  where $\lesssim$ means up to a constant that only depends on $d$.
\end{lemma}

\begin{proof}[Proof of Lemma~\ref{L:bound}]

  \step 1 Reduction to an estimate for $\aa^{\bb,0}$.

  We claim that
  \begin{equation*}
    \frac{\aa(\bb)}{1-\aa(\bb)\nabla\nabla G_T(\aa,\bb,\bb)}\leq
    (1+\nabla\nabla G_T(\aa^{\bb,0},\bb,\bb))^2
  \end{equation*}
  For the argument let $\aa\in\Omega$ and $\bb\in\B^d$ be fixed. With
  the shorthand
  notation introduced in \eqref{eq:44}, the claim reads
  \begin{equation}\label{eq:bound:step1}
    \frac{a_0}{1-a_0g(a_0)}\leq (1+g(0))^2.
  \end{equation}
  For $a_0=0$ the
  statement is trivial. For $a_0>0$ consider the function
  \begin{equation*}
    f(a):=\frac{1}{a}g(a)-g^2(a),
  \end{equation*}
  with help of which the left-hand side in \eqref{eq:bound:step1} can be written as $\frac{g(a_0)}{f(a_0)}$.  The function $f$ is non-negative and decreasing, as can be seen by
  combining the inequality $0<g(a)<\frac{1}{a}$ from Lemma~\ref{L:ODE} with the identity
  $f'(a)=g(a)(g^2(a)-\frac{1}{a^2}+g^2(a)-\frac{1}{a}g(a))$ which
  follows from \eqref{eq:L:osc:1b}. The latter also implies that $g(1)=\frac{g(0)}{1+g(0)}$ and
  thus $f(1)=g(1)(1-g(1))=\frac{g(0)}{(1+g(0))^2}$.
  Hence, 
  \begin{equation*}
    \frac{a_0}{1-a_0g(a_0)}=\frac{g(a_0)}{f(a_0)}\leq \frac{g(a_0)}{f(1)}=
    (1+g(0))^2\frac{g(a_0)}{g(0)}\leq (1+g(0))^2;
  \end{equation*}
  in the last step we used in addition that $g(a_0)\leq g(0)$ which is a
  consequence of \eqref{eq:L:osc:1b}.
  \medskip

  \step 2 Conclusion.

  To complete the argument we only need to show that
  \begin{equation}\label{eq:48}
    \nabla\nabla G_T(\aa^{\bb,0},\bb,\bb)\lesssim \omega_0(\aa,\bb).
  \end{equation}
  For simplicity set $\aa_0:=\aa^{\bb,0}$. Note that
  $\omega_0(\aa,\bb)=\omega(\aa_0,\bb)$. From \eqref{eq:46} we obtain
  \begin{eqnarray*}
    \nabla\nabla G_T(\aa_0,\bb,\bb)&\stackrel{\eqref{eq:46}}{\geq}& \sum_{\bb'\in\B^d}\aa_0(\bb')\left(\nabla\nabla
    G_T(\aa_0,\bb',\bb)\right)^2\stackrel{\eqref{eq:coercivity}}{\gtrsim} \sum_{\bb'\in\B^d}\omega^{-1}(\aa_0,\bb')\left(\nabla\nabla
    G_T(\aa_0,\bb',\bb)\right)^2\\
&\geq& \omega^{-1}(\aa_0,\bb)\left(\nabla\nabla
    G_T(\aa_0,\bb,\bb)\right)^2.
  \end{eqnarray*}
  Dividing both sides by $\omega^{-1}(\aa_0,\bb)\nabla\nabla
  G_T(\aa_0,\bb,\bb)$ yields \eqref{eq:48}.
\end{proof}

\subsection{Proof of Lemma~\ref{lem:coercivity-phys} and Lemma~\ref{lem:coercivity-prob}}

\begin{proof}[Proof of Lemma~\ref{lem:coercivity-phys}]
Fix for a moment $\aa\in\Omega$. For $\bb\in\B^d$ with
$\dist_{\aa}(x_{\bb},y_{\bb})<\infty$, let $\pi_{\aa}(\bb)$ denote a shortest
open path that connects $x_{\bb}$ and $y_{\bb}$, i.e.
\begin{equation*}
  \dist_{\aa}(x_{\bb},y_{\bb})=\sum_{\bb'\in\pi_{\aa}(\bb)}\frac{1}{\aa(\bb')}.
\end{equation*}
Thanks to the
triangle inequality and the Cauchy-Schwarz inequality we
have
\begin{eqnarray*}
  |\nabla u(\bb)|&\leq&\sum_{\bb'\in\pi(\bb)}|\nabla
  u(\bb')|\leq
  \left(\sum_{\bb'\in\pi_{\aa}(\bb)}\frac{1}{\aa(\bb')}\right)^{\frac{1}{2}}\,\left(\sum_{\bb'\in\pi_{\aa}(\bb)}|\nabla
    u(\bb')|^2\aa(\bb')\right)^{\frac{1}{2}}\\
  &=&\dist_{\aa}^\frac{1}{2}(x_{\bb},y_{\bb})\left(\sum_{\bb'\in\pi_{\aa}(\bb)}|\nabla
    u(\bb')|^2\aa(\bb')\right)^{\frac{1}{2}}.
\end{eqnarray*}
Hence, using the convention $\frac{1}{\infty}=0$, we conclude that for all
$\bb\in\B^d$ and $\aa\in\Omega$:
\begin{equation}\label{eq:11}
  \dist^{-p}_{\aa}(x_{\bb},y_{\bb})|\nabla u(\bb)|^2\,\leq\,\dist^{1-p}_{\aa}(x_{\bb},y_{\bb})\sum_{\bb'\in\pi_{\aa}(\bb)}|\nabla
    u(\bb')|^2\aa(\bb').
\end{equation}
We drop the ``$\aa$'' in the notation from now on. Summation of
\eqref{eq:11} in
$\bb\in\B^d$ yields
\begin{eqnarray*}
  \sum_{\bb\in\B^d}\dist^{-p}(x_{\bb},y_{\bb})|\nabla u(\bb)|^2
  &\leq&\sum_{\bb\in\B^d}\sum_{\bb'\in\pi(\bb)}\dist^{1-p}(x_{\bb},y_{\bb})|\nabla
  u(\bb')|^2\aa(\bb')\\
  &=&\sum_{\bb'\in\B^d}\sum_{\bb\in\B^d\text{ with }\atop\pi(\bb)\ni\bb'}\dist^{1-p}(x_{\bb},y_{\bb})|\nabla
  u(\bb')|^2\aa(\bb').
\end{eqnarray*}
Since $\pi(\bb)$ is a shortest path, and because $\aa\leq 1$, we
have $\dist(x_{\bb},y_{\bb})\geq |x_{\bb}-x_{\bb'}|+1$ for all
  $\bb,\bb'\in\B^d$ with $\bb'\in\pi(\bb)$. Combined  with the
  previous estimate we get
\begin{eqnarray*}
  \sum_{\bb\in\B^d}\dist^{-p}(x_{\bb},y_{\bb})|\nabla u(\bb)|^2
  &\leq&\sum_{\bb'\in\B^d}\sum_{\bb\in\B^d\text{ with }\atop\pi(\bb)\ni\bb'}(|x_{\bb}-x_{\bb'}|+1)^{1-p}|\nabla
  u(\bb')|^2\aa(\bb')\\
  &\leq&C(d,p)\,\sum_{\bb'\in\B^d}|\nabla
  u(\bb')|^2\aa(\bb').
\end{eqnarray*}
\end{proof}

\begin{proof}[Proof of Lemma~\ref{lem:coercivity-prob}]
  Fix $\bb\in\B^d$. For $L\in\N$ consider the indicator function
  \begin{equation}\label{eq:13}
    \chi_L(\aa):=
    \left\{\begin{aligned}
        &1&&\text{if }L\leq\dist_{\aa}(x_{\bb},y_{\bb})<2L,\\
        &0&&\text{else}.    
      \end{aligned}\right.
  \end{equation}
  With the convention $\frac{1}{\infty}=0$, we have
  \begin{equation}\label{eq:15}
    \sum_{k=0}^\infty\chi_{2^k}(\aa)\dist^{-p}_{\aa}(x_{\bb},y_{\bb})=\dist^{-p}_{\aa}(x_{\bb},y_{\bb})
  \end{equation}
  for all $\aa\in\Omega$. In the following we drop ``$\aa$''
  in the notation. We recall \eqref{eq:11} in the form of
  \begin{equation}\label{eq:14}
    \chi_L\dist^{-p}(x_{\bb},y_{\bb})|\nabla u(\bb)|^2
    \,\leq\,\chi_L\dist^{1-p}(x_{\bb},y_{\bb})\sum_{\bb'\in\pi(\bb)}|\nabla
    u(\bb')|^2\aa(\bb').
  \end{equation}
  From $\aa\leq 1$ and $\dist(x_{\bb},y_{\bb})<2L$ for $\chi_L\neq 0$, cf. \eqref{eq:13},
  we learn that $\pi(\bb)$ is contained in the box $Q_{2L}(x_{\bb})$.
  Hence, \eqref{eq:14} turns into
  \begin{equation*}
    \chi_L\dist^{-p}(x_{\bb},y_{\bb})|\nabla u(\bb)|^2
    \,\stackrel{\eqref{eq:13}}{\leq}\,\chi_L L^{1-p}\sum_{\bb'\in Q_{2L}(x_{\bb})}|\nabla
    u(\bb')|^2\aa(\bb').
  \end{equation*}
  We take the expectation on both sides and appeal to stationarity:
  \begin{eqnarray*}
    \expec{\chi_L\dist^{-p}(x_{\bb},y_{\bb})|\nabla u(\bb)|^2}
    &\leq&L^{1-p}\sum_{\bb'\in Q_{2L}(x_{\bb})}\expec{\chi_L|\nabla
    u(\bb')|^2\aa(\bb')}\\
    &\stackrel{\chi_L\leq 1}{\leq}&L^{1-p}\sum_{x\in
      B_{2L}(x_{\bb})}\sum_{\bb'=\{x,x+e_i\}\atop i=1,\ldots,d}\expec{|\nabla
    u(\bb')|^2\aa(\bb')}\\
    &\stackrel{\text{stationarity}}{\leq}& L^{1-p}|B_{2L}(0)|\sum_{\bb'=\{0,e_i\}\atop i=1,\ldots,d}\expec{|\nabla
    u(\bb')|^2\aa(\bb')}.
  \end{eqnarray*}
  Using $1+d-p<0$ we get
  \begin{eqnarray*}
    \expec{\dist^{-p}(x_{\bb},y_{\bb})|\nabla u(\bb)|^2}
    &\stackrel{\eqref{eq:15}}{=}&\sum_{k=0}^\infty\expec{\chi_{2^k}\dist^{-p}(x_{\bb},y_{\bb})|\nabla u(\bb)|^2}\\
    &\leq& C(p,d)\,\sum_{\bb'=\{0,e_i\}\atop i=1,\ldots,d}\expec{|\nabla
    u(\bb')|^2\aa(\bb')}.
  \end{eqnarray*}
\end{proof}
   
\subsection{Proof of Proposition~\ref{P1}  -- Green's function estimates}\label{S:green}
We first establish an estimate  for
the Green's function itself:
\begin{lemma}\label{L:P1:2}
  Let $d\geq2$ and consider $u,f\in\ell^1(\Z^d)$ with
  \begin{equation}\label{L:P1:2-1}
    \nabla^*\aa\nabla u=f\qquad\text{in }\Z^d.
  \end{equation}
  Then for all $\frac{2d}{d+2}<p<2$, $R\geq 1$ and $x_0\in\Z^d$ we have
  \begin{equation}\label{L:4:a}
    \sum_{x\in B_R(x_0)}|u(x)-\bar u|\lesssim
    C\,R^{2}\sum_{x\in\Z^d}|f(x)|.
  \end{equation}
  Here, $\bar u:=\frac{1}{|B_R(x_0)|}\sum_{x\in B_R(x_0)}u(x)$
  denotes the average of $u$ on $B_R(x_0)$,
  $C:=C(\aa,Q_R(x_0),\tfrac{p}{2-p})$, and $\lesssim$
  means $\leq$ up to a  constant that only depends on
  $d$ and $p$.
\end{lemma}
\begin{proof}[Proof of Lemma~\ref{L:P1:2}]
   W.~l.~o.~g. we assume $\sum_{\Z^d}|f|=1$ and $R\in\N$. To shorten
  the notation  we write $B_R$ and $Q_R$ for $B_R(x_0)$ and
  $Q_R(x_0)$, respectively.
  Let $M(u)$ denote a median of $u$ on $B_R$, i.~e.
  \begin{equation*}
    |\{u\geq M(u)\}\cap B_R|,    |\{u\leq M(u)\}\cap B_R|\geq\frac{1}{2}|B_R|.
  \end{equation*}
  By Jensen's inequality we have $|\bar u-M(u)|\leq\frac{1}{|B_R|}\sum_{B_R}|u-M(u)|$,
  so that it suffices to prove for $v:=u-M(u)$ the estimate
  \begin{equation*}
    \sum_{B_R}|v|\lesssim C\,R^2\sum_{\Z^d} |f|
    =C\,R^2.
  \end{equation*}
  For $0\leq M<\infty$ consider the cut-off version of $v$
  \begin{equation*}
    v_M:=\max\{\min\{v,M\},0\}.
  \end{equation*}
  Then $v_M$ satisfies 
  \begin{eqnarray*}
    \sum_{\B^d}\nabla v_M\,\aa\nabla v_M=\sum_{\B^d}\nabla u\,\aa\nabla v_M.
  \end{eqnarray*}
  Since $u\in\ell^1(\Z^d)$ (by assumption) and $v_M\in\ell^\infty(\Z^d)$ (by construction), we may integrate by parts:
  \begin{equation*}
   \sum_{\B^d}\nabla u\,\aa\nabla v_M=
   \sum_{\Z^d} v_M\,\nabla^*\aa\nabla u=\sum_{\Z^d} fv_M\leq
   M\sum_{\Z^d}|f|=M.
  \end{equation*}
  Hence,
  \begin{equation}\label{eq:L1:1}
    \sum_{\B^d}\nabla v_M\,\aa\nabla v_M\leq M.
  \end{equation}
  Set $p^*=\frac{pd}{d-p}$ and $q^*:=\frac{p^*}{p^*-1}$. 
  By construction  we have $|\{v_M=0\}\cap B_R|=|\{v_M\leq 0\}\cap B_R|\geq \frac{1}{2}|B_R|$. 
  Hence,   the Sobolev-Poincar\'e inequality yields
  \begin{equation*}
    \left(R^{-d}\sum_{B_R}|v_M|^{p^*}\right)^{\frac{1}{p^*}}
    \lesssim R\left(R^{-d}\sum_{Q_R}|\nabla v_M|^p\right)^{\frac{1}{p}}.
  \end{equation*}
  Lemma \ref{lem:coercivity-phys} combined with H\"older's inequality
  with exponents $(\frac{2}{2-p},\frac{2}{p})$ yields
  \begin{eqnarray}\nonumber
    \left(R^{-d}\sum_{Q_R}|\nabla v_M|^{p}\right)^{\frac{1}{p}}
    &=&
    \left(R^{-d}\sum_{Q_R}\omega^{\frac{p}{2}}|\nabla
      v_M|^{p}\omega^{-\frac{p}{2}}\right)^{\frac{1}{p}}\\\nonumber
    &\leq&
    \left(R^{-d}\sum_{Q_R}\omega^{\frac{p}{2-p}}\right)^{\frac{2-p}{2p}}
    \left(R^{-d}\sum_{Q_R}|\nabla
      v_M|^2\omega^{-1}\right)^{\frac{1}{2}}\\\label{eq:coercivity2}
    &\stackrel{\text{Lemma~\ref{lem:coercivity-phys}}}{\lesssim}&
    C^{\frac{1}{2}}\ 
    \left(R^{-d}\sum_{\B^d}\nabla v_M\,\aa\nabla
      v_M\right)^{\frac{1}{2}},
  \end{eqnarray}
  so that
  \begin{align}
    \label{eq:L1:2}
      \left(R^{-d}\sum_{B_R}|v_M|^{p^*}\right)^{\frac{1}{p^*}}\, 
      \lesssim\, C^{\frac{1}{2}}R\left(R^{-d}
      \sum_{\B^d}\nabla v_M\,\aa\nabla v_M\right)^{\frac{1}{2}}
      \stackrel{\eqref{eq:L1:1}}{\lesssim}\,(C
      R^{2-d}M)^{\frac{1}{2}}.
  \end{align}
  Next we use Chebyshev's inequality in the form of
  \begin{equation*}
    M\left(R^{-d}|\{\,v>M\,\}\cap B_R|\right)^{\frac{1}{p^*}}
    \lesssim    \left(R^{-d}\sum_{B_R}|v_M|^{p^*}\right)^{\frac{1}{p^*}}.
  \end{equation*}
  With \eqref{eq:L1:2} we get
  \begin{equation*}
    R^{-d}|\{\,v>M\,\}\cap B_R|\lesssim C^{\frac{p^*}{2}}\,R^{(2-d)\frac{p^*}{2}}M^{-\frac{p^*}{2}},
  \end{equation*}
  which upgrades by symmetry to
  \begin{equation*}
    R^{-d}|\{\,|v|>M\,\}\cap B_R|\lesssim C^{\frac{p^*}{2}}\,R^{(2-d)\frac{p^*}{2}}M^{-\frac{p^*}{2}}.
  \end{equation*}
  Since $p>\frac{2d}{d+2}$ (by assumption), we have $\frac{p^*}{2}>1$ and the ``wedding cake
  formula'' for $M:=CR^{2-d}$ yields
  \begin{eqnarray*}
    R^{-d}\sum_{B_R}|v|&=&\int_0^\infty R^{-d}|\{\,|v|>M'\,\}\cap
    B_R|\,dM'\,\lesssim\, M+\int_M^\infty R^{-d}|\{\,|v|>M'\,\}\cap B_R|\,dM'\\
    &\lesssim& M + C^{\frac{p^*}{2}}R^{(2-d)\frac{p^*}{2}}M^{1-\frac{p^*}{2}}\,\lesssim\,CR^{2-d}.
\end{eqnarray*}
  \end{proof}
A careful Caccioppoli estimate combined with the previous lemma yields:
\begin{lemma}\label{L:P1:2b}
  Let $d\geq2$, $x_0\in\Z^d$ and $R\geq 1$. Consider $f\geq 0$ and $u$
  related as
  \begin{equation}\label{eq:L2:1}
    \nabla^*\aa\nabla u=-f\qquad\text{in }B_{2R}(x_0).
  \end{equation}
  Then for $\frac{2d}{d+2}<p<2$ we have
  \begin{equation}\label{L2:1}
    \left(R^{-d}\sum_{Q_R(x_0)}|R\nabla u|^p\right)^{\frac{1}{p}}
		\lesssim C^{\frac{\alpha}{2}}\,\left(R^{-d}\sum_{B_{2R}(x_0)}|u|+\left(R^{2-d}\sum_{B_{2R}(x_0)} fu_-\right)^{\frac{1}{2}}\right),
  \end{equation}
	where $u_-:=\max\{-u,0\}$ denotes the negative part of $u$, 
  $C:=C(\aa, Q_{2R}(x_0),\tfrac{p}{2-p})$,
  $\alpha:=2\frac{p^*-1}{p^*-2}$ and $p^*:=\frac{dp}{d-p}$. Here
  $\lesssim$ stands for $\leq$ up to a constant that only depends on
  $p$ and $d$
\end{lemma}
\begin{proof}[Proof of Lemma~\ref{L:P1:2b}]
  \step 1 Caccioppoli estimate.
  
  We claim that  for every cut-off function $\eta$ that is supported in $B_{2R-1}(x_0)$ (so that in particular $\nabla\eta=0$
  outside of $Q_{2R}(x_0)$) we have
  \begin{equation}\label{eq:cacc2}
    \left(R^{-d}\sum_{\B^d}|R\nabla(u\eta)|^p\right)^{\frac{1}{p}}\lesssim
    C^{\frac{1}{2}}\left(R^{2-d}\sum_{\Z^d}fu_-\eta^2+R^{-d}\sum_{\bb\in\B^d}u(x_{\bb})u(y_{\bb})|R\nabla\eta(\bb)|^2\aa(\bb)   \right)^{\frac{1}{2}}.
  \end{equation}
 Indeed, we get with Lemma~\ref{lem:coercivity-phys} (using an argument
 similar to \eqref{eq:coercivity2}):
 \begin{equation*}
   \left(R^{-d}\sum_{\B^d}|R\nabla(u\eta)|^p\right)^{\frac{1}{p}}=\left(R^{-d}\sum_{Q_{2R}(x_0)}|R\nabla(u\eta)|^p\right)^{\frac{1}{p}}\lesssim
   \,C^{\frac{1}{2}}\left(R^{-d}\sum_{\B^d}|R\nabla(u\eta)|^2\aa\right)^{\frac{1}{2}},
 \end{equation*}
 Combined with the elementary identity
 \begin{equation*}
   |\nabla(u\eta)(\bb)|^2=\nabla u(\bb)\nabla(u\eta^2)(\bb)+u(x_{\bb})u(y_{\bb})|\nabla\eta(\bb)|^2,
 \end{equation*}
 the equation for $u$, and the fact that $-fu\eta^2\leq
 fu_-\eta^2$ (here we use $f\geq 0$), the claimed estimate
 \eqref{eq:cacc2} follows.
 \medskip


  \step 2 Conclusion.

  Set  $\theta:=\frac{\alpha-1}{\alpha}$ and note that $\alpha$ is
  defined in such a way that for the considered range of $p$ we have
  \begin{equation}\label{eq:L2:03}
    \frac{1}{2}=\theta
    \frac{1}{p^{*}}+(1-\theta)\qquad\text{and}\qquad 2(1-\theta)<1.
  \end{equation}
  As we shall see below in Step~3, there exists a cut-off function
  $\eta$ with $\eta=1$ in $B_{R+1}(x_0)$ and $\eta=0$ outside of $B_{2R-1}(x_0)$, such that
  \begin{multline}
    \label{eq:L2:02}
    \left(R^{-d}\sum_{\bb\in\B^d}|u(x_{\bb})||u(y_{\bb})||R\nabla\eta(\bb)|^2\right)^{\frac{1}{2}}\,\lesssim\,\left(R^{-d}\sum_{\Z^d} |u\eta|^{p^*}\right)^{\frac{\theta}{p^*}}
    \left(R^{-d}\sum_{B_{2R}(x_0)}|u|\right)^{1-\theta}\\
    +\left(R^{-d}\sum_{\Z^d} |u\eta|^{p^*}\right)^{\frac{1}{2p^*}}
    \left(R^{-d}\sum_{B_{2R}(x_0)}|u|\right)^{\frac{1}{2}}.
  \end{multline}
  Let us explain the right-hand side of this estimate. While the first
  term on the right-hand side would also appear in the continuum case (i.e.~when $\Z^d$ is
  replaced by $\R^d$), the second term is an error term coming from discreteness. In
  fact, it is of lower order: A sharp look at \eqref{eq:16} below shows
  that \eqref{eq:L2:02} holds with  the vanishing factor $R^{-\epsilon}$ (for
  some $\epsilon>0$ only depending on $p$ and $d$) in front of the
  second term on the right-hand side.

  By combining this estimate with the Gagliardo-Nirenberg-Sobolev
  inequality on $\Z^d$,\linebreak i.e.~$\left(R^{-d}\sum_{\Z^d}|u\eta|^{p^*}\right)^{\frac{1}{p^*}}
  \,\lesssim\,\left(R^{-d}\sum_{\B^d}|R\nabla(u\eta)|^p\right)^{\frac{1}{p}}$,
  and two applications of Youngs' inequality, we find that for all $\delta>0$ there exists a constant $C(\delta)>0$ only depending on $\delta$, $p$ and $d$, such that
  \begin{equation*}
    \begin{split}
      &\left(CR^{-d}\sum_{\bb\in\B^d}|u(x_{\bb})||u(y_{\bb})|(\nabla\eta(\bb))^2\aa(\bb)\right)^{\frac{1}{2}}\\
      &\qquad\qquad\leq\,\delta\left(R^{-d}\sum_{\B^d}
        |R\nabla(u\eta)|^{p}\right)^{\frac{1}{p}} +
      C(\delta)\left(C^{\frac{1}{2(1-\theta)}}R^{-d}\sum_{B_{2R}(x_0)}|u|
        + CR^{-d}\sum_{B_{2R}(x_0)}|u|\right)\\
      &\qquad\qquad\stackrel{2(1-\theta)<1}{\leq}\,\delta\left(R^{-d}\sum_{\B^d}
        |R\nabla(u\eta)|^{p}\right)^{\frac{1}{p}} +
      2C(\delta)C^{\frac{1}{2(1-\theta)}}R^{-d}\sum_{B_{2R}(x_0)}|u|.
    \end{split}
  \end{equation*}
  We combine this estimate with \eqref{eq:cacc2} and absorb the first
  term on the right-hand side of the previous estimate into the
  left-hand side of \eqref{eq:cacc2}. Since $\nabla(\eta u)=\nabla u$ in $Q_{R}(x_0)$ this yields \eqref{L2:1}.
  \smallskip

  \step 3 Proof of \eqref{eq:L2:02}.

  We first construct a suitable cut-off function $\eta$ for $B_{R+1}(x_0)$ in
  $B_{2R-1}(x_0)$. W.~l.~o.~g. we assume that $x_0=0$. Recall that $\alpha=2\frac{p^*-1}{p^*-2}$.
  For $t\geq 0$ set
  \begin{equation*}
    \tilde\eta(t):=\max\{1-2\max\{\tfrac{t}{R+1}-1,0\},0\}^\alpha,
  \end{equation*}
  and define
  \begin{equation}
   \label{eq:cutoff}
    \eta(x):=\prod_{i=1}^d\tilde\eta(|x_i|).
  \end{equation}
  Using the relation $\alpha-1=\theta\alpha$, cf. \eqref{eq:L2:03}, it is straightforward to check that $\eta$ satisfies for all edges $\bb$ with
  $|\nabla\eta(\bb)|>0$:
  \begin{align}\label{eq:L2:3a}
    &R|\nabla\eta(\bb)|\lesssim
    \begin{cases}
      \min\{\eta^{\theta}(x_{\bb}),\eta^\theta(y_{\bb})\} & \text{if }\min\{\eta(x_{\bb}),\eta(y_{\bb})>0\},\\
      R^{1-\alpha} & \text{if }\min\{\eta(x_{\bb}),\eta(y_{\bb})\}=0.
    \end{cases}
  \end{align}
  Now we turn to \eqref{eq:L2:02}. We split the sum into a ``interior''
  and a ``boundary'' contribution:
  \begin{multline*}
    \sum_{\bb\in\B^d} |u(x_{\bb})||u(y_{\bb})|(\nabla\eta(\bb))^2\\
    = \sum_{\bb\in A_{\text{int}}} |u(x_{\bb})||u(y_{\bb})|(\nabla\eta(\bb))^2
    +\sum_{\bb\in A_{\text{bound}}} |u(x_{\bb})||u(y_{\bb})|(\nabla\eta(\bb))^2,
  \end{multline*}
      where
  \begin{eqnarray*}
    A_{\text{int}}&:=&\{\,\bb\,:\,|\nabla\eta(\bb)|>0\,\text{ and
    }\,\min\{\eta(x_{\bb}),\eta(y_{\bb})\}>0\,\},\\
    A_{\text{bound}}&:=&\{\,\bb\,:\,|\nabla\eta(\bb)|>0\,\text{ and }\,\min\{\eta(x_{\bb}),\eta(y_{\bb})\}=0\,\}.
  \end{eqnarray*}
  For $A_{\text{int}}$ we get with  \eqref{eq:L2:3a}, Young's
  inequality, and H\"older's inequality with exponents
  $(p^*\frac{1}{2\theta},\frac{1}{2(1-\theta)})$:
  \begin{equation}\label{eq:L2:021}
    \begin{aligned}
      &R^{-d}\sum_{\bb\in A_{\text{int}}}
      |u(x_{\bb})||u(y_{\bb})||R\nabla\eta(\bb)|^2\,\lesssim\,
      R^{-d}\sum_{\Z^d}u^2\eta^{2\theta}\\
      &\qquad =\,R^{-d}\sum_{\Z^d}(u\eta)^{2\theta}u^{2(1-\theta)}
      \, \leq\,
      \left(R^{-d}\sum_{\Z^d}
        (u\eta)^{p^*}\right)^{\frac{2\theta}{p^*}}\left(R^{-d}\sum_{B_{2R}}
        |u|\right)^{2(1-\theta)}.
    \end{aligned}
  \end{equation}
  Next we treat $A_{\text{bound}}$, which is an error term coming from discreteness. By the definition of $A_{\text{bound}}$ the cut-off function
  $\eta$ vanishes at one and only one of the two sites adjacent to $\bb\in
  A_{\text{bound}}$. Given $\bb\in A_{\text{bound}}$ we denote by $\tilde x_{\bb}$ (resp. $\tilde
  y_{\bb}$) the site adjacent to $\bb$ with $\eta(\tilde x_{\bb})=0$ (resp.
  $\eta(\tilde y_{\bb})\neq 0$), so that
  \begin{equation*}
    R^{-d}\sum_{\bb\in A_{\text{bound}}}|u(x_{\bb})||u(y_{\bb})||R\nabla\eta(\bb)|^2=R^{1-d}\sum_{\bb\in A_{\text{bound}}}|u(\tilde
    x_{\bb})||u(\tilde y_{\bb})|\eta(\tilde y_{\bb})|R\nabla\eta(\bb)|.
  \end{equation*}
  We combine this with \eqref{eq:L2:3a}, H\"older's inequality with
  exponents $(p^*,q^*:=\frac{p^*}{p^*-1})$, and the discrete $\ell^{1}$-$\ell^{q^*}$-estimate:
  \begin{align}\nonumber
    &R^{-d}\sum_{\bb\in
      A_{\text{bound}}}|u(x_{\bb})||u(y_{\bb})||R\nabla\eta(\bb)|^2\,\lesssim
    \,R^{2-d-\alpha}\sum_{\bb\in A_{\text{bound}}}|u(\tilde  x_{\bb})||u(\tilde y_{\bb})|\eta(\tilde y_{\bb})\\\nonumber
    &\qquad\leq\,
    R^{2-d-\alpha}\left(\sum_{B_{2R}}|u\eta|^{p^*}\right)^{\frac{1}{p^*}}\left(\sum_{B_{2R}}|u|^{q^*}\right)^{\frac{1}{q^*}}\,\leq\,
    R^{2-d-\alpha}\left(\sum_{B_{2R}}|u\eta|^{p^*}\right)^{\frac{1}{p^*}}\sum_{B_{2R}}|u|\\\label{eq:16}
    &\qquad=\,
    R^{\frac{d}{p^*}-2-\alpha}\left(R^{-d}\sum_{B_{2R}}|u\eta|^{p^*}\right)^{\frac{1}{p^*}}\left(R^{-d}\sum_{B_{2R}}|u|\right).
  \end{align}
  From the definition of $\alpha$ and $p^*$, and the fact that
  $\alpha>2$, we deduce that the exponent $\frac{d}{p^*}-2-\alpha$ is negative. Together with  \eqref{eq:L2:021} the desired estimate \eqref{eq:L2:02} follows.
\end{proof}

Now we are ready to prove Proposition~\ref{P1}. We distinguish the
cases $k\geq 1$ and $k=0$.
\begin{proof}[Proof of Proposition~\ref{P1}]
\step 1 Argument for $k\geq 1$.

  For brevity set $R:=2^{k-1}R_0$ and recall that $A_k=Q_{2
    R}(0)\setminus Q_{R}(0)$. We cover the annulus $A_k$ by boxes $Q_{\frac{R}{2}}(x_0)$, $x_0\in X_R\subset\Z^d$, such that
  \begin{equation}\label{eq:4}
    A_k\subset\bigcup\limits_{x_0\in X_R}Q_{\frac{R}{2}}(x_0)\subset \bigcup\limits_{x_0\in X_R}Q_{R}(x_0)\subset
    Q_{3R}(0)\setminus \{0\}.
  \end{equation}
  Since the diameter of the annulus and the side length of the boxes
  are comparable, we may choose $X_R$ such that its
  cardinality is bounded by a constant only depending on $d$.
  Since in addition we have for $x_0\in X_R$ the inequality
  $C(\aa,Q_R(x_0),\tfrac{p}{2-p})\lesssim C(\aa, Q_{3R}(0),\tfrac{p}{2-p})$ (thanks to the
  third inclusion in \eqref{eq:4}), it suffices to prove
  \begin{equation*}
    \left(R^{-d}\sum_{\bb\in Q_{\frac{R}{2}}(x_0)}|\nabla
      G_T(\aa,\bb,0)|^p\right)^{\frac{1}{p}}\lesssim
    C^{\frac{\beta}{2}}\,R^{1-d},\qquad\text{where } C:=C(\aa,
    Q_{R}(x_0),\tfrac{p}{2-p}),
  \end{equation*}
  for each $x_0\in X_R$ separately.
We use the shorthand $G_T(x):=
  G_T(\aa,x,0)$ and set $\bar G_T:=\frac{1}{|B_R(x_0)|}\sum_{x\in
    B_R(x_0)}G_T(x)$. In view of \eqref{eq:D:G},  $u(x):=G_T(x)-\bar
  G_T$ satisfies \eqref{L:P1:2-1} with $f=\delta-\frac{1}{T}G_T$.
  Since
  \begin{equation}\label{eq:17}
    \sum_{\Z^d}|\delta-\frac{1}{T}G_T|\leq 1+\frac{1}{T}\sum_{\Z^d}G_T(x)=2,
  \end{equation} 
  Lemma~\ref{L:P1:2} yields
  \begin{equation}\label{eq:P1:1}
    R^{-d}\sum_{B_{R}(x_0)}|u|\lesssim C^{\frac{1}{2}p^*}\,R^{2-d}.
  \end{equation}
  Thanks to the third inclusion in \eqref{eq:4} we have $0\notin
  B_R(x_0)$, and thus $u$ satisfies \eqref{eq:L2:1} with
  $f=\frac{1}{T}G_T$ (with $B_{2R}(x_0)$ replaced by $B_{R}(x_0)$).
  Hence, Lemma~\ref{L:P1:2b} yields
  \begin{eqnarray}
    \begin{split}
      \label{eq:lemma2app}
      \left(R^{p-d}\sum_{ Q_{\frac{R}{2}}(x_0)}|\nabla G_T|^p\right)^{\frac{1}{p}}&=\left(R^{p-d}\sum_{ Q_{\frac{R}{2}}(x_0)}|\nabla u|^p\right)^{\frac{1}{p}}\\
      &\lesssim C^{\frac{1}{2}\alpha}\,
      R^{-d}\sum_{B_{R}(x_0)}
      |u|\,+\,C^{\frac{1}{2}\alpha}\,\left(R^{2-d}\sum_{B_R(x_0)}
        \frac{1}{T}G_Tu_-\right)^{\frac{1}{2}}\\
      &\stackrel{\eqref{eq:P1:1}}{\lesssim} C^{\frac{1}{2}(\alpha+p^*)}\,R^{2-d}\,+\,C^{\frac{1}{2}\alpha}\,\left(R^{2-d}\sum_{B_R(x_0)}
        \frac{1}{T}G_Tu_-\right)^{\frac{1}{2}}.
    \end{split}
  \end{eqnarray}
  Regarding the second term on the right-hand side we only need to
  show \begin{equation}\label{eq:P1:2}
    \frac{1}{T}\sum_{B_R(x_0)} G_Tu_-\lesssim C^{p^*} R^{2-d}.
  \end{equation} 
   We note that
  $(G_T-\bar G_T)(G_T-\bar G_T)_-\leq 0$, so that
  \begin{eqnarray*}
    \frac{1}{T}\sum_{B_R(x_0)} G_T u_-
    &=&\frac{1}{T}\sum_{B_R(x_0)}(G_T-\bar G_T+\bar G_T)(G_T-\bar G_T)_-
    \leq \frac{1}{T}\bar G_T\sum_{B_{R}(x_0)}|G_T-\bar G_T|.
  \end{eqnarray*}
  Combined with  \eqref{eq:P1:1} and the inequality $\frac{1}{T}\bar G_T\lesssim R^{-d}\frac{1}{T}\sum_{B_{R}(x_0)}G_T\leq R^{-d}$,
  \eqref{eq:P1:2} follows.
\medskip

\step 2 Argument for $k=0$.
  Fix $\aa\in\Omega$. For brevity set $G_T(x):=G_T(\aa,x,0)$ and $\bar G_T:=\frac{1}{|B_{2R_0}(0)|}\sum_{x\in
    B_{2R_0}(0)}G_T(x)$. By the discrete 
	$\ell^1$-$\ell^{p}$-estimate and the elementary inequality
        $|\nabla G_T(\bb)|\leq |G_T(x_{\bb})-\bar G_T|+|G_T(y_{\bb})-\bar
        G_T|$ we have
        \begin{equation*}
          \left(\frac{1}{|Q_{R_0}(0)|}\sum_{\bb\in Q_{R_0}(0)}|\nabla
            G_T(\bb)|^p\right)^{\frac{1}{p}}\lesssim \sum_{       B_{2R_0}(0)}|G_T-\bar G_T|.
        \end{equation*}
        As in Step~1 an application of Lemma~\ref{L:P1:2} yields
	\begin{equation*}
          \sum_{B_{2R_0}(0)}|G_T-\bar
          G_T|\lesssim C^{\frac{p^*}{2}}(\aa,Q_{2R_0}(0),\tfrac{p}{2-p})
          \,R_0^{2}.
        \end{equation*}
        Since $R_0^{2}\sim R_0^{1-d}$ and because the exponent of the
        constant satisfies $\frac{p^*}{2}\leq\frac{\beta}{2}$, the desired
        estimate follows.
      \end{proof}

\subsection{Proof of  Lemma~\ref{L:CEP}}
In order to deal with the failure of the Leibniz rule we will appeal to a number of discrete 
estimates, which are stated in Lemma~\ref{lem:1} below. As already mentioned, we
replace the missing uniform ellipticity of $\aa$ by the coercivity
estimate of Lemma~\ref{lem:coercivity-prob} which makes use of the
weight $\omega$ defined in \eqref{D:omega}. Morally speaking it plays the
role of $\frac{1}{\lambda_0}$ in \eqref{eq:29}. In view of Assumption (A2) all
moments of $\omega$ are bounded, i.~e. $\expec{\omega^{k}}\lesssim 1$, where $\lesssim$ means $\leq$ up to a  constant that
only depends on $k$, $p$, $\Lambda$ and $d$. 
We split the proof of Lemma~\ref{L:CEP} into the following two inequalities:
\begin{align}
  \label{eq:36}
  \expec{|\nabla\phi(\bb)|^{2p+1}}^{\frac{2p+2}{2p+1}}\ &\lesssim
  \sum_{\bb'=\{0,e_i\}\atop
    i=1,\ldots,d}\expec{|\nabla(\phi^{p+1})(\bb')|^2\aa(\bb')},\\
  \label{eq:37}
  \sum_{\bb'=\{0,e_i\}\atop
    i=1,\ldots,d}\expec{|\nabla(\phi^{p+1})(\bb')|^2\aa(\bb')}\
  &\lesssim \expec{\phi^{2p}(x)}.
\end{align}
Here and below we write $\phi$ instead of $\phi_T$ for simplicity.
Note that due to stationarity the left-hand side of \eqref{eq:36} and
the right-hand side of \eqref{eq:37} do not depend on $\bb\in\B^d$
(resp. $x\in\Z^d$). Therefore, we suppress these arguments in the
following. We start with \eqref{eq:36}. We smuggle in $\omega$
by appealing to H\"older's inequality with exponent
$\frac{2p+2}{2p+1}$ and exploit that all moments of $\omega$ are
bounded by Assumption (A2):
\begin{align*}
 \expec{|\nabla\phi|^{2p+1}}^{\frac{2p+2}{2p+1}}
 \lesssim\expec{|\nabla\phi|^{2p+2}\omega^{-1}}.
\end{align*} 
We combine \eqref{eq:discrete} in the form of
$|\nabla\phi(\bb)|^{2p+2}\lesssim(\frac{\phi^p(x_{\bb})+\phi^p(y_{\bb})}{2})^2|\nabla\phi(\bb)|^2$ (where
we use that $p$ is even) with the discrete version of the Leibniz rule
$F^p\nabla F=\frac{1}{p+1}\nabla(F^{p+1})$, see \eqref{eq:CorLeibniz2} in
Corollary~\ref{C:leibniz} below:
\begin{equation}\label{eq:35}
\expec{|\nabla\phi|^{2p+2}\omega^{-1}}
 \lesssim \expec{|\nabla(\phi^{p+1})|^2\omega^{-1}}.
\end{equation}
Now \eqref{eq:36} follows from  the coercivity estimate of
Lemma~\ref{lem:coercivity-prob}. 

Next we prove \eqref{eq:37}. The discrete version of the Leibniz rule $|\nabla(F^{p+1})|^2=\frac{(p+1)^2}{(2p+1)}\nabla
F\nabla(F^{2p+1})$ (see Lemma~\ref{lem:1} (ii)) yields
\begin{equation*}
 \sum_{\bb'=\{0,e_i\}\atop
    i=1,\ldots,d}\expec{|\nabla(\phi^{p+1})(\bb')|^2\aa(\bb')}
\lesssim \sum_{\bb'=\{0,e_i\}\atop i=1,\ldots,d}\expec{\nabla\phi(\bb')\aa(\bb')\nabla(\phi^{2p+1})(\bb')}.
\end{equation*}
By stationarity and the modified corrector equation
\eqref{eq:cor-modified} we have
\begin{eqnarray*}
  \lefteqn{\sum_{\bb'=\{0,e_i\}\atop
      i=1,\ldots,d}\expec{\nabla\phi(\bb')\aa(\bb')\nabla(\phi^{2p+1})(\bb')}=\expec{(\nabla^*\aa\nabla\phi)\
      \phi^{2p+1}}}&& \\
    &=& -\frac{1}{T}\expec{\phi^{2(p+1)}}-\sum_{\bb'=\{0,e_i\}\atop
      i=1,\ldots,d}\expec{\nabla\phi^{2p+1}(\bb')\aa(\bb')e (\bb')}\\
    &\leq & \sum_{\bb'=\{0,e_i\}\atop
      i=1,\ldots,d}\expec{|\nabla(\phi^{2p+1})(\bb')|\aa(\bb')},
\end{eqnarray*}
where for the last inequality we use that $\phi^{2(p+1)}\geq 0$ and
$|e |=1$. By Corollary~\ref{C:leibniz} and Young's inequality we get for any $\epsilon>0$
\begin{eqnarray*}
  \sum_{\bb'=\{0,e_i\}\atop i=1,\ldots,d}\expec{|\nabla(\phi^{2p+1})(\bb')|\aa(\bb')}
  &\stackrel{\eqref{eq:CorLeibniz1}}{\lesssim}&
  \epsilon\sum_{\bb'=\{0,e_i\}\atop i=1,\ldots,d}\expec{|\nabla\phi(\bb')|^2\left(\tfrac{\phi^p(x_{\bb'})+\phi^p(y_{\bb'})}{2}\right)^2\aa(\bb')}
  \\
  &&+\frac{1}{\epsilon}
 \sum_{\bb'=\{0,e_i\}\atop i=1,\ldots,d}\expec{\left(\tfrac{\phi^p(x_{\bb'})+\phi^p(y_{\bb'})}{2}\right)^2}\\
 &\stackrel{\eqref{eq:CorLeibniz2}}{\lesssim}&
 \epsilon\sum_{\bb'=\{0,e_i\}\atop
      i=1,\ldots,d}\expec{|\nabla(\phi^{p+1})(\bb')|^2\aa(\bb')}
 +\frac{1}{\epsilon}\expec{\phi^{2p}}.
\end{eqnarray*}
Since we may choose $\epsilon>0$ as small as we wish, the first term
on the right-hand side can be absorbed into the left-hand side of
\eqref{eq:37} and the claim follows.

\section*{Acknowledgments}
We thank Artem Sapozhnikov for stimulating discussions on percolation
models. Stefan Neukamm was
partially supported by ERC-2010-AdG no.267802 AnaMultiScale. Most of
this work was done while all three authors were employed at the
Max-Planck-Institute for Mathematics in the Sciences, Leipzig.

\appendix

\subsection*{Appendix: Replacements of the Leibniz rule for the discrete derivative}
\begin{lemma}
  \label{lem:1}
  Let $F$ be a scalar function on $\Z^d$ and $\bb\in\B^d$.
  \begin{enumerate}[(i)]
  \item Assume that $p\in2\N$. Then we have
    \begin{equation*}
     |\nabla(F^{p+1})(\bb)|\sim|\nabla F(\bb)|\frac{F^p(x_{\bb})+F^p(y_{\bb})}{2}.
    \end{equation*}
  \item For every integer $p$ we have
   \begin{align*}
     |\nabla(F^{p+1})(\bb)|^2\lesssim\nabla F(\bb)\nabla(F^{2p+1})(\bb).
   \end{align*}
 \end{enumerate}
 Here $\lesssim$ (resp. $\sim$) means up to a  constant
 that only depends on $p$.
\end{lemma}
\begin{proof}[Proof of Lemma~\ref{lem:1}]
Let $x,y\in\Z^d$ denote the vertices with $\bb=\{x,y\}$ and
$y-x\in\{e_1,\ldots,e_d\}$ so that $\nabla F(\bb)=F(y)-F(x)$.
\medskip

 {\it Proof of part (i).} The statement $"\lesssim"$ is equivalent
 to \cite[Equation (5.29)]{GO1} and is proven there.
 Concerning $\gtrsim$ we appeal to \cite[Equation (5.28)]{GO1}.
 From that equation we learn that
  \begin{equation*}
  \nabla(F^{p+1})(\bb)\nabla F(b)\gtrsim\frac{F^p(x_{\bb})+F^p(y_{\bb})}{2}|\nabla F(\bb)|^2.
  \end{equation*}
  By dividing by $|\nabla F(\bb)|$ one immediately finds the claimed result.
  \medskip

 {\it Proof of part (ii).} We have to distinguish two cases.

 First case: $F(x),F(y)\geq 0$ or $F(x),F(y)\leq 0$. It suffices to show the statement for $F(x),F(y)\geq 0$, 
since then the case $F(x),F(y)\leq 0$ follows by symmetry. We have to prove that 
\begin{align*}
  (F^{p+1}(y)-F^{p+1}(x))^2\lesssim (F(y)-F(x))(F^{2p+1}(y)-F^{2p+1}(x)).
 \end{align*}
 By symmetry and and scale invariance, it suffices to show the elementary inequality
 \begin{equation}\label{eq:1}
    \forall f\geq 0:\ \ \      
(1-f^{p+1})^2\leq c (1-f)(1-f^{2p+1}),
  \end{equation} 
where $c>0$ only depends on $p$. We omit its proof for the sake of brevity.

 Second case: $F(x)\leq 0,F(y)\geq 0$ or $F(x)\geq 0,F(y)\leq 0$.
  It suffices to show the statement for $F(x)\leq 0, F(y)\geq 0$, since then
  the case $F(x)\geq 0, F(y)\leq 0$ follows by symmetry. We have to prove
that
  \begin{align*}
  (F^{p+1}(y)-F^{p+1}(x))^2\lesssim (F(y)-F(x))(F^{2p+1}(y)-F^{2p+1}(x))
 \end{align*}
 or equivalently
 \begin{align*}
   &F^{2(p+1)}(y)+F^{2(p+1)}(x)-2F^{p+1}(y)F^{p+1}(x)\\
  &\qquad\qquad \lesssim F^{2p+2}(y)+F^{2p+2}(x)-F(x)F^{2p+1}(y)-F(y)F^{2p+1}(x).
 \end{align*}
Note that since $2p+1$ is an odd integer, the last two terms on the right
hand side
of the above inequality are positive. Hence, it suffices to prove that
 \begin{align*}
  F^{2(p+1)}(y)+F^{2(p+1)}(x)-2F^{p+1}(y)F^{p+1}(x)\lesssim
  F^{2p+2}(y)+F^{2p+2}(x),
 \end{align*}
which follows due to $-2F^{p+1}(y)F^{p+1}(x)\leq F^{2p+2}(y)+F^{2p+2}(x)$.
\end{proof}

In the course of proving our main result we will use the discrete Leibniz
rule,
(i) in the above lemma, in the following form.

\begin{corollary}\label{C:leibniz}
 For every scalar function $F$, every bond $\bb$ and every even integer
$p$ we have
  \begin{align}
    \label{eq:CorLeibniz1}
     |\nabla(F^{2p+1})(\bb)|&\lesssim |\nabla
     F(\bb)|\left(\frac{F^p(x_{\bb})+F^p(y_{\bb})}{2}\right)^2,\\
    \label{eq:CorLeibniz2}
     |\nabla F(\bb)|^2\left(\frac{F^p(x_{\bb})+F^p(y_{\bb})}{2}\right)^2&\lesssim|\nabla(F^{p+1})(\bb)|^2.
  \end{align}
 Here $\lesssim$ means up to a  constant
 that only depends on $p$.

 \end{corollary}

\end{document}